\documentclass{amsart}
\usepackage[a4paper,twoside,inner=2cm,outer=2cm]{geometry}
\usepackage{amssymb,amsmath,amscd}
\usepackage{tikz}
\usetikzlibrary{matrix,arrows}
\usepackage{graphicx}  
\usepackage{float}  
\usepackage{multirow}  
\usepackage{booktabs}  
\usepackage{subcaption}  
\usepackage[hidelinks]{hyperref}

\theoremstyle{definition}
\newtheorem{dfn}{Definition}[section]
\newtheorem{exm}[dfn]{Example}

\theoremstyle{remark}
\newtheorem{rmk}[dfn]{Remark}
\numberwithin{equation}{subsection}
\theoremstyle{plain}
\newtheorem{thm}[dfn]{Theorem}
\newtheorem{prop}[dfn]{Proposition}
\newtheorem{lem}[dfn]{Lemma}
\newtheorem{cor}[dfn]{Corollary}

\newcommand*{\Cdot}{\raisebox{-1ex}{\scalebox{3}{$\cdot$}}}
\renewcommand{\leq}{\leqslant}

\renewcommand{\setminus}{\smallsetminus}

\newcommand{\Z}{\mathbb{Z}}
\newcommand{\Q}{\mathbb{Q}}
\newcommand{\R}{\mathbb{R}}
\newcommand{\C}{\mathbb{C}}
\renewcommand{\P}{\mathbb{P}}

\title{Localization of equivariant cohomology rings of real Grassmannians}



\author[He]{\bfseries Chen He}

\address{
	Yau Mathematical Sciences Center \\ 
	Tsinghua University \\ 
	Beijing\\
	P.R.~China}

\email{che@math.tsinghua.edu.cn}

\begin{document}


\vspace{18mm} \setcounter{page}{1} \thispagestyle{empty}

\begin{abstract}
We use localization method to understand the rational equivariant cohomology rings of real Grassmannians and oriented Grassmannians, then relate this to the Leray-Borel description which says the ring generators are equivariant Pontryagin classes, Euler classes in even dimension, and one more new type of classes in odd dimension, as stated by Casian and Kodama. We give additive basis in terms of equivariant characteristic polynomials and equivariant Schubert/canonical classes. We also calculate Poincar\'e series, equivariant Littlewood-Richardson coefficients and equivariant characteristic numbers. Since all these Grassmannians with torus actions are equivariantly formal, many results for equivariant cohomology have similar statements for ordinary cohomology. 
\end{abstract}

\maketitle

\section{Introduction}
\vskip 15pt

The study of homology and cohomology of real and complex Grassmannian was initially based on Ehresmann's work \cite{Eh37} in terms of Schubert cells. The relations between Schubert classes and characteristic classes were given by Chern \cite{Ch48} for real Grassmannians in $\Z/2$ coefficients. Using the techniques of spectral sequences of fibre bundles, Leray \cite{Le46,Le49} and Borel \cite{Bo53A,Bo53B} gave a unified way to describe the cohomology rings of certain homogeneous spaces in terms of characteristic classes. 

These pioneering work applies not only to ordinary cohomology but also to equivariant cohomology. For compact connected Lie group $G$ and its connected closed subgroup $H$ of the same rank, the maximal torus $T$ acts from the left of the homogeneous space $G/H$ with the Leray-Borel description of its equivariant cohomology in $\Q$ coefficients:
\[
H^*_T(G/H)=\mathbb{S}\mathfrak{t}^*\otimes_{(\mathbb{S}\mathfrak{t}^*)^{W_G}} (\mathbb{S}\mathfrak{t}^*)^{W_H}
\]
where $\mathbb{S}\mathfrak{t}^*$ is the symmetric algebra of the dual Lie algebra $\mathfrak{t}^*$, and $W_G,W_H$ are the Weyl groups of $G$ and $H$. 

For example, if we consider complex flag varieties as quotients of unitary groups $U(n)$, then the Weyl groups are products of symmetric groups and the Weyl group invariants $(\mathbb{S}\mathfrak{t}^*)^{W_G}, (\mathbb{S}\mathfrak{t}^*)^{W_H}$ consist of symmetric polynomials in appropriate variables. In topological terminology, these invariants are polynomials of equivariant Chern classes of canonical bundles, with relations from Whitney product formula.  

Similarly, for even dimensional oriented Grassmannians viewed as quotients of $SO(n)$, whose Weyl groups act by permutations and sign changes on polynomials in appropriate variables, the Leray-Borel description means theirs equivariant cohomology rings are generated by equivariant Pontryagin classes and equivariant Euler classes of canonical bundles and complementary bundles, with relations from Whitney product formula and square of Euler classes as top Pontryagin classes. Notice that oriented Grassmannians are natural $2$-fold covers of real Grassmannians, then we can identify the equivariant cohomology of real Grassmannians as $\Z/2$-invariant of the equivariant cohomology of oriented Grassmannians. Due to the lack of preferred orientations for subspaces in $\R^n$, there is no Euler class of canonical bundle or complementary bundle over real Grassmannians. These facts give even dimensional real Grassmannians their Leray-Borel description of equivariant cohomology rings generated by equivariant Pontryagin classes, with relations from Whitney product formula. This special case of Leray-Borel description for real Grassmannians are stated in Casian\&Kodama \cite{CK}. For odd dimensional oriented Grassmannians $SO(2n+2)/SO(2k+1)\times SO(2n-2k+1)$ in which $SO(2n+2)$ and $SO(2k+1)\times SO(2n-2k+1)$ have different ranks of maximal tori, similar Leray-Borel description was obtained by Takeuchi \cite{Ta62}. 

In this paper, we will use localization methods to understand the rational equivariant cohomology rings of real and oriented Grassmannians and re-derive the Leray-Borel description. We will give additive basis both in characteristic polynomials and in canonical classes, compute Poincar\'e series and characteristic numbers, and try to relate these results with Schubert calculus on real Grassmannians.

Alternatively, Sadykov \cite{Sa17}, Carlson \cite{Carl} and He \cite{HeB} have given short derivations of the Leray-Borel-Takeuchi descriptions of the rational cohomology rings of real and oriented Grassmannians. 

\textbf{Acknowledgement} This work was part of the author's PhD thesis. The author would like to thank Victor Guillemin and Jonathan Weitsman for guidance. The author would also like to thank Jeffrey Carlson for many useful discussions and for the references of Leray and Takeuchi.

\vskip 20pt
\section{Torus actions, equivariant cohomology}
\vskip 15pt
In this section, we recall some basics of equivariant cohomology for torus actions.

\subsection{Torus actions and isotropy weights at fixed points}\label{subsec:T-action}
Throughout the paper, a manifold $M$ is always assumed to be smooth, compact, but not necessarily oriented nor connected. Let torus $T$ act on a manifold $M$, we will denote $M^T$ as the fixed-point set. For any point $p$ in a connected component $C$ of $M^T$, there is the \textbf{isotropy representation} of $T$ on the tangent space $T_p M$, which splits into weighted spaces $T_p M = V_0 \oplus V_{[\alpha_1]} \oplus \cdots \oplus V_{[\alpha_r]}$ where the non-zero distinct weights $[\alpha_1],\ldots,[\alpha_r] \in \mathfrak{t}^*_\Z/{\pm 1}$ are determined only up to signs (If $M$ has a $T$-invariant stable almost complex structure, then those weights are determined without ambiguity of signs). Comparing with the tangent-normal splitting $T_p M = T_p C \oplus N_p C$, we get that $T_p C = V_0$ and $N_p C = V_{[\alpha_1]} \oplus \cdots \oplus V_{[\alpha_r]}$. Since $N_p C = V_{[\alpha_1]} \oplus \cdots \oplus V_{[\alpha_r]}$ is of even dimension, the dimensions of $M$ and of the components of $M^T$ will be of the same parity. If $\mathrm{dim}\, M$ is even, the smallest possible components of $M^T$ could be isolated points.  If $\mathrm{dim}\, M$ is odd, the smallest possible components of $M^T$ could be isolated circles. Moreover, since $T$ acts on the normal space $N_p C$ by rotation, this gives the normal space $N_p C$ an orientation. 

For any subtorus $K$ of $T$, we get two more actions automatically: the \textbf{sub-action} of $K$ on $M$ and the \textbf{residual action} of $T/K$ on $M^K$.

\subsection{Equivariant cohomology}
Let torus $T=(S^1)^n$ act on a manifold $M$. The $T$-equivariant cohomology of $M$ is defined using the Borel construction $H^*_{T} (M) = H^*((ET\times M) / T)$, where $ET=(S^\infty)^{n}$ with $BT=ET/T=(\C P^\infty)^{n}$ and the coefficient ring will usually be $\Q$ throughout the paper, unless otherwise mentioned to be $\Z$. By this definition, if we denote $\mathfrak{t}^*$ as the dual Lie algebra of $T$, then $H^*_T (pt)=H^*(ET / T)=H^*((\C P^\infty)^{n})=\mathbb{S}\mathfrak{t}^*$ is a polynomial ring $\Q[\alpha_1,\dots,\alpha_n]$ or $\Z[\alpha_1,\ldots,\alpha_n]$ under the identification that $c_1(\gamma_i\rightarrow BT)=\alpha_i$, where $\gamma_i\rightarrow BT$ is the canonical complex line bundle on the $i$-th $\C P^\infty$-component  of $BT$ and $\alpha_i \in \mathfrak{t}^*_\Z$ is the integral weight dual to the $i$-th $S^1$-component of $T$. The trivial map $\iota: M \rightarrow pt$ induces a homomorphism $\iota^*: H^*_T(pt)\rightarrow H^*_T(M)$ and hence makes $H^*_T(M)$ a $H^*_T (pt)$-module.

For a $T$-equivariant complex, oriented or real vector bundle $V$ over $M$, we can define the equivariant Chern, Euler or Pontryagin classes respectively as $c^T(V)=c((ET\times V) / T)\,,e^T(V)=e((ET\times V) / T)\,,p^T(V)=p((ET\times V) / T)\in H^*_T(M,\Z)$.

The famous Atiyah-Bott-Berline-Vergne(ABBV) localization formula says:

\begin{thm}[ABBV Localization Formula, \cite{BV83,AB84}]\label{ABBV}
	On an oriented $T$-manifold $M$, an equivariant cohomology class $\omega \in H^*_T (M)$ can be integrated as
	\[
	\int_M \omega = \sum_{C \subseteq M^T} \int_C \frac{\omega|_C}{e^T(NC)}
	\]
	where the summation is taken for every component $C \subseteq M^T$ with normal bundle $NC$ and equivariant Euler class $e^T(NC)$.
\end{thm}

Inspired by this localization theorem, one can hope for more connections between the manifold $M$ and its fixed-point set $M^T$, if $H^*_T(M)$ is actually a free $H^*_T (pt)$-module. 

\begin{dfn}
	An action of $T$ on $M$ is \textbf{equivariantly formal} if $H^*_T (M)$ is a free $H^*_T (pt)$-module. 
\end{dfn}

Using the techniques of spectral sequences, equivariant formality has various equivalent expressions. 

\begin{thm}[Equivalences of equivariant formality, \cite{AP93} pp.\,210 Thm\,3.10.4, \cite{GGK02} pp.\,206-207] \label{thm:formal}
Let torus $T$ act on a manifold $M$, the following conditions about equivariant cohomology are equivalent:
\begin{enumerate}
	\item The $T$-action is equivariantly formal, i.e. $H^*_T (M)$ is a free $H^*_T (pt)$-module
	\item The Leray-Serre sequence of the fibration $M \hookrightarrow (M \times ET) /T \rightarrow BT$ collapses with $E_\infty = E_2 = H^*(BT)\otimes H^*(M)$
	\item $H^*_T(M)\cong H^*_T (pt)\otimes H^*(M)$ as $H^*_T (pt)$-module
	\item $H^*_T(M)\rightarrow H^*(M)$, defined as the restriction to the fibre $M$, is surjective
    \item Any additive basis of $H^*(M)$ can be lifted to $H^*_T(M)$, hence give an additive $H^*_T (pt)$-basis for $H^*_T(M)$
	\item $\sum \mathrm{dim}\,H^*(M^T) = \sum \mathrm{dim}\,H^*(M)$.
\end{enumerate}
\end{thm}
\begin{rmk}
	The equivalences among (2)(3)(4)(5) are direct applications of the Leray-Hirsch theorem (which works not only for $\Q$ coefficients, but also for $\Z$ coefficients if $H^*(M,\Z)$ is a free $\Z$-module). For the equivalence to the remaining conditions (1)(6) in $\Q$ coefficients, see the cited references.
\end{rmk}

\begin{rmk}
    When the Betti numbers of $M$ and $M^T$ are known, the equality $\sum \mathrm{dim}\,H^*(M^T) = \sum \mathrm{dim}\,H^*(M)$ is a handy way to verify the equivariant formality.
\end{rmk}

\begin{rmk}
	The fibre inclusion $M \hookrightarrow (M \times ET) /T$ induces a homomorphism $H^*_T(M) \rightarrow H^*(M)$ factoring through $\Q\otimes_{H^*_T (pt)} H^*_T(M)$, where $\Q$ has a $H^*_T (pt)=\Q[\alpha_1,\ldots,\alpha_n]$-algebra structure from the constant-term morphism $\Q[\alpha_1,\ldots,\alpha_n]\rightarrow \Q: f(\alpha_1,\ldots,\alpha_n)\mapsto f(0)$. When the $T$-action is equivariantly formal, i.e. $H^*_T(M)\cong H^*_T (pt)\otimes_\Q H^*(M)$, then we can recover $H^*(M)$ as $\Q\otimes_{H^*_T (pt)} H^*_T(M)$.
\end{rmk}

\vskip 20pt
\section{GKM-type theorems}
\vskip 15pt
In this section, we recall the Chang-Skjelbred lemma, the GKM theorem and the author's recent work on GKM-type theorems for possibly non-orientable manifolds in odd dimensions.

\subsection{Chang-Skjelbred lemma and GKM condition}

Given an action $T\curvearrowright M$, for every point $p\in M$, its stabilizer is defined as $T_p=\{t \in T \mid t\cdot p = p\}$, and its orbit is $\mathcal{O}_p\cong T/T_p$. Let's set the $1$-skeleton $M_1=\{p \mid \textup{dim}\,\mathcal{O}_p\leq 1\}$, the union of $1$-dimensional orbits and fixed points. If $H^*_T (M)$ is a free $H^*_T (pt)$-module, i.e. the action is equivariantly formal, Chang and Skjelbred \cite{CS74} proved that $H^*_T (M)$ only depends on the fixed-point set $M^T$ and the $1$-skeleton $M_1$.

\begin{thm}[Chang-Skjelbred Lemma, \cite{CS74}]\label{Chang}
	If an action $T\curvearrowright M$ is equivariantly formal, then
	\[
	H^*_T(M) \cong H^*_T(M_1)\hookrightarrow H^*_T(M^T).
	\]
\end{thm}

This Lemma enables one to describe the equivariant cohomology $H^*_T(M)$ as a sub-ring of $H^*_T(M^T)$, subject to certain algebraic relations determined by the $1$-skeleton $M_1$. To apply the Chang-Skjelbred Lemma, we will follow Goresky, Kottwitz and MacPherson's idea to start with the smallest fixed-point set $M^T$ and $1$-skeleton $M_1$.

\begin{dfn}[GKM condition in either even or odd dimensions]\label{GKMCond}
	An action $T\curvearrowright M$ is \textbf{GKM} if
	\begin{itemize}
		\item[(1)] The fixed-point set $M^T$ consists of isolated points or isolated circles. 
		\item[(2)] At each fixed point $p\in M^T$, the non-zero weights $[\alpha_1],\ldots,[\alpha_n] \in \mathfrak{t}^*_\Z/{\pm 1}$ of the isotropy $T$-representation $T\curvearrowright T_pM$ are pair-wise independent.
	\end{itemize}
\end{dfn}

\begin{rmk}
	As mentioned in the Subsection\,\ref{subsec:T-action}, the dimensions of $M$ and of $M^T$ have the same parity. Condition\,(1) in Definition\,\ref{GKMCond} means that $M^T$ consists of isolated points when $M$ is even dimensional or $M^T$ consists of isolated circles when $M$ is odd dimensional.
\end{rmk}

\begin{rmk}
	The GKM condition is equivalent to requiring the $1$-skeleton $M_1$ to be 2-dimensional when $M$ is even dimensional or 3-dimensional when $M$ is odd dimensional.
\end{rmk}

\subsection{GKM theorem: the even dimensional case}

Goresky, Kottwitz and MacPherson \cite{GKM98} considered torus actions on algebraic varieties where the fixed-point set $M^T$ is finite and the $1$-skeleton $M_1$ is a union of spheres $S^2$. They proved that the cohomology $H^*_T(M)$ can be described in terms of congruence relations on a graph determined by the $1$-skeleton $M_1$. Goertsches and Mare \cite{GM14} observed that GKM theory also works in non-orientable case by adding $\R P_{[\alpha]}^2$ components to the $1$-skeleton $M_1$.

\begin{dfn}[GKM graph in even dimension]
	If an action $T\curvearrowright M^{2m}$ is GKM, then its \textbf{GKM graph} consists of 
	\begin{description}
		\item[Vertices] A $\bullet$ for each fixed point in $M^T$
		\item[Edges$\,\&\,$Weights] A solid edge with weight $[\alpha]$ for each $S^1_{[\alpha]}$ joining two $\bullet$'s representing its two fixed points, and a dotted edge with weight $[\beta]$ for each $\R P^2_{[\beta]}$ joining a $\bullet$ to an empty vertex.
	\end{description}
\end{dfn}

\begin{thm}[GKM theorem in even dimension, \cite{GKM98} pp.\,26 Thm\,1.2.2, \cite{GM14} pp.\,7 Thm\,3.6]\label{thm:EvenGKM}
	If the action of a torus $T$ on a (possibly non-orientable) manifold $M^{2m}$ is equivariantly formal and GKM, then we can construct its GKM graph $\Gamma$, with vertex set $V=M^T$ and weighted edge set $E$, moreover the equivariant cohomology has a graphic description
	\[
	H^*_T(M) = \big\{ f: V\rightarrow \mathbb{S}\mathfrak{t}^* \mid f_p \equiv f_q \mod{\alpha} \quad \mbox{for each solid edge $\overline{pq}$ with weight $\alpha$ 
		in $E$}\big\}.
	\]
\end{thm}

\begin{rmk}
	Since $H^*_T(\R P_{[\beta]}^2)=H^*_T(pt)$ does not contribute to congruence relations, we can erase all the dotted edges of $\R P_{[\beta]}^2$ and only keep the solid edges of $S^2_{[\alpha]}$ to construct an \textbf{effective} GKM graph, which does not necessarily have the same number of edges for every vertex.
\end{rmk}

\subsection{GKM-type theorem: the odd dimensional case}

In odd dimensions, we can also consider smallest-dimensional $1$-skeleton, i.e. $M_1$ is $3$-dimensional. Then according to Chang-Skjelbred lemma, the localization boils down to understanding $S^1$-actions on $3$-manifolds, studied by the author \cite{He17}.

Similar to the even dimensional case, we can construct a graph for each odd dimensional $T$-manifold under the GKM condition\,\ref{GKMCond}.

\begin{dfn}[GKM graph in odd dimension]
	If an action $T\curvearrowright M^{2m+1}$ (possibly non-orientable) is GKM, then its \textbf{GKM graph} consists of
	\begin{description}
		\item[Vertices] There will be two types of vertices.
		\begin{itemize}
			\item [$\circ$] for each fixed circle $C \subset M^T$.
			\item [$\Box$] for each 3d connected component $N^3_{[\alpha]}$ in $M^{T_{[\alpha]}}$ of some codimension-$1$ subtorus $T_{[\alpha]}$ which has Lie algebra $\mathfrak{t_\alpha}\subset \mathfrak{t}$ annihilated by $\alpha$. The $\Box$ is then weighted with $[\alpha]$.
		\end{itemize}
		
		\item[Edges] An edge joins a $(\Box,N)$ to a $(\circ,C)$, if the 3d manifold $N$ contains the fixed circle $C$ and hence is a connected component of $M^{T_{[\alpha]}}$ for an isotropy weight $[\alpha]$ of $C$. There are no edges directly joining $\circ$ to $\circ$, nor $\Box$ to $\Box$.
		
	\end{description}
\end{dfn}

In order to derive a GKM-type graphic description of $H_{T}^*(M^{2m+1})$, we need to fix in advance an orientation $\theta_i$ for each fixed circle $C_i \subseteq M^T$, and also fix an orientation for each orientable $M^{T_{[\alpha]}} \subseteq M_1$.

\begin{thm}[A GKM-type theorem in odd dimension, \cite{HeA}]\label{thm:OddGKM}
	If the action of a torus $T$ on (possibly non-orientable) manifold $M^{2n+1}$ is equivariantly formal and GKM, then we can construct its GKM graph $\Gamma$, with two types of vertex sets $V_\circ$ and $V_\Box$ and edge set $E$. An element of the equivariant cohomology $H^*_T(M)$ can be written as:
	\[
	(P,Q\theta): V_\circ \longrightarrow \mathbb{S}\mathfrak{t}^* \oplus \mathbb{S}\mathfrak{t}^* \theta
	\]
	where $\theta$ is the generator of $H^1(S^1)$,
	under the relations that for each $\Box$ representing a 3d component $N$ of some $M^{T_{[\alpha]}}$ and the neighbour $\circ$'s representing the fixed circles $C_1,\ldots,C_k$ on this component,
	\begin{itemize}
		\item if $N$ is non-orientable,
		\begin{equation*}
		P_{C_1}\equiv P_{C_2}\equiv \cdots\equiv P_{C_k} \mod{\alpha} 
		\end{equation*}
		\item if $N$ is orientable,
		\begin{equation*}
		P_{C_1}\equiv P_{C_2}\equiv \cdots\equiv P_{C_k} \mbox{ and } \sum_{i=1}^k \pm Q_{C_i}\equiv 0 \mod{\alpha} 
		\end{equation*}
		where the sign for each $Q_{C_i}$ is specified by comparing the prechosen orientation $\theta_i$ with the induced orientation of $N$ on $C_i$.
	\end{itemize}
\end{thm}

\begin{rmk}
	As discussed in \cite{HeA}, different choices of orientations from $C_i \subseteq M^T$ and from orientable $M^{T_{[\alpha]}} \subseteq M_1$ give the isomorphic equivariant cohomology. When $M$ has a $T$-invariant stable almost complex structure, then the isotropy weights $\alpha$ can be determined without ambiguity of signs. Moreover, $M^T$ and $M^{T_{\alpha}} \subseteq M_1$ are equipped with induced stable almost complex structures, hence are oriented canonically.
\end{rmk}

\begin{rmk}
	To describe the $\mathbb{S}\mathfrak{t}^*$-algebra structure of $H^*_T(M^{2m+1})$, it is convenient to write an element $(P,Q\theta)$ as $(P_C+Q_C\theta)_{C\subset M^T}$. Note $\theta^2=0$, then $(P_C+Q_C\theta)_{C\subset M^T} + (\bar{P}_C+\bar{Q}_C\theta)_{C\subset M^T}=([P_C+\bar{P}_C]+[Q_C+\bar{Q}_C]\theta)_{C\subset M^T}$, and $(P_C+Q_C\theta)_{C\subset M^T} \Cdot (\bar{P}_C+\bar{Q}_C\theta)_{C\subset M^T}=([P_C \bar{P}_C]+[P_C\bar{Q}_C+\bar{P}_C Q_C]\theta)_{C\subset M^T}$. For any polynomial $R \in \mathbb{S}\mathfrak{t}^*$, we have $R \Cdot (P_C+Q_C\theta)_{C\subset M^T} = (R P_C+R Q_C\theta)_{C\subset M^T}$.
\end{rmk}

\vskip 20pt
\section{Equivariant cohomology rings of complex Grassmannians}
\vskip 15pt
In this section, we recall the GKM description and Leray-Borel description of equivariant cohomology rings of complex Grassmannians, together with the characteristic basis and canonical basis of the additive structure. We use the notation $G_k(\C^n)$ for the Grassmannian of $k$-dimensional complex subspaces in $\C^n$.

\subsection{GKM description of complex Grassmannians}
As shown by Guillemin, Holm and Zara \cite{GHZ06}, for compact connected group $G$ and its closed connected subgroup $H$ of the same rank as $G$, the homogeneous space $G/H$ is GKM and equivariantly formal under the left action of maximal torus $T$, hence has a graphic description for its equivariant cohomology. For example, the GKM graph of $T^n \curvearrowright U(n)/(U(k)\times U(n-k))$ is the Johnson graph $J(n,k)$, of which each vertex is a $k$-element subset $S \subseteq \{1,\ldots,n\}$ and two vertices $S,S'$ are joined by an edge if they differ by one element. For later use in the case of real Grassmannians, we will give an explicit description for the $1$-skeleta of complex Grassmannians.

\begin{prop}[$1$-skeleta of complex Grassmannians]
	The $T^n$-action on $G_k(\C^n)$ has $\binom{n}{k}$ fixed points of the form $\oplus_{i\in S} \C_i$, where $S$ is a $k$-element subset of $\{1,2,\ldots,n\}$ and $\C_i$ is the $i$-th component of $\C^n$. The isotropy weights at $\oplus_{i\in S} \C_i$ are $\{\alpha_j-\alpha_i \mid i\in S,j\not \in S\}$, and join $\oplus_{i'\in S} \C_{i'}$ to $\oplus_{i'\in (S\setminus\{i\})\cup \{j\}} \C_{i'}$ via $\{(\oplus_{i'\in S\setminus\{i\}} \C_{i'})\oplus L \mid L \in \P(\C_i\oplus\C_j)\}\cong \C P^1$ in the $1$-skeleton.
\end{prop}
\begin{proof}
	$T^n$ acts on $\C^n$ linearly by $(t_1,\ldots,t_n)\cdot (z_1,\ldots,z_n)=(t_1 z_1,\ldots,t_1 z_n)$ and hence induces an action on $G_k(\C^n)$ by mapping every $k$-dimensional subspace $V$ to $t\cdot V$ for each $t\in T^n$. A fixed point $V$ of the $T$-action on $G_k(\C^n)$ is exactly a $k$-dimensional sub-representation of $\C^n= \oplus_{i=1}^n \C_i$. Since $\oplus_{i=1}^n \C_i$ has distinct weights $\alpha_1,\ldots,\alpha_n$, a $k$-dimensional sub-representation is of the form $\oplus_{i\in S} \C_i$ for $S$, a $k$-element subset of $\{1,2,\ldots,n\}$.
	
	To understand the isotropy weights at $\oplus_{i \in S} \C_i \in G_k(\C^n)$, notice that the tangent space of $G_k(\C^n)$ at $\oplus_{i \in S} \C_i$ is $Hom_\C(\oplus_{i \in S} \C_i, \oplus_{j \not \in S} \C_j)\cong (\oplus_{i \in S} \C_i)^* \otimes_\C (\oplus_{j \not \in S} \C_j)\cong\oplus_{i \in S} \oplus_{j \not \in S} (\C_i^* \otimes_\C \C_j)$ with pair-wise independent weights $\{\alpha_j-\alpha_i \mid i \in S \text{ and } j \not \in S\}$.
	
	The finiteness of fixed points and pair-wise independence of isotropy weights at every fixed point verifies that the $T^n$ action on $G_k(\C^n)$ is GKM. Moreover, for a $k$-element subset $S\subseteq \{1,2,\ldots,n\}$ and a weight $\alpha_j-\alpha_i$ with $j\not \in S,\, i \in S$, the $2$-sphere 
	\[
	\big\{(\underset{i'\in S\setminus \{i\}}{\oplus} \C_{i'}) \oplus L \mid L \in \P(\C_i\oplus \C_j) \big\} \cong \C P^1
	\]
	connects the $T$-fixed point $\oplus_{i'\in S} \C_{i'}$ with the $T$-fixed point $\oplus_{i'\in (S\setminus\{i\})\cup \{j\}} \C_{i'}$, and is fixed by the corank-$1$ subtorus torus $T_{\alpha_j-\alpha_i}$ whose Lie algebra is $\textup{Ker}(\alpha_j-\alpha_i)$.
\end{proof}

Since $G_k(\C^n)$ doesn't have odd-degree cells, the canonical $T^n$ action on $G_k(\C^n)$ is equivariantly formal in $\Z$ coefficients and of course in $\Q$ coefficients. Then we can apply the even dimensional GKM theorem to $G_k(\C^n)$ using congruence relations on the Johnson graph $J(n,k)$. 

\begin{thm}[GKM description of complex Grassmannians, \cite{GZ01}]\label{thm:GKMcplxGrass}
	Let $\mathcal{S}$ be the collection of $k$-element subsets of $\{1,2,\ldots,n\}$, then the equivariant cohomology of the $T^n$ action on $G_k(\C^n)$ is
	\[
	H^*_{T}(G_k(\C^n),\Q)=\big\{f:\mathcal{S}\rightarrow \Q[\alpha_1,\ldots,\alpha_n] \mid f_{S} \equiv f_{S'} \mod \alpha_j-\alpha_i \quad \text{for $S,S' \in \mathcal{S}$ with $S\cup\{j\}=S'\cup\{i\}$}\big\}.
	\]
\end{thm}

Using Morse theory on graphs, Guillemin and Zara analysed additive basis for equivariant cohomology of GKM manifolds.

\begin{thm}[Canonical basis of complex Grassmannians, \cite{GZ03}]\label{thm:CplxSchub}
	There is a self-indexing Morse function on $\mathcal{S}$
	\[
	\phi: \mathcal{S} \longrightarrow \R : S \longmapsto 2(\sum_{i\in S} i) -k(k+1)
	\]
	and a canonical class $\tau_S \in H^{\phi(S)}_{T}(G_k(\C^n),\Q)$ for each $S\in \mathcal{S}$ such that
	\begin{enumerate}
		\item $\tau_S$ is supported upward, i.e. $\tau_S(S')=0$ if $\phi(S')\leq \phi(S)$
		\item $\tau_S(S)=\prod' (\alpha_j - \alpha_i)$ where the product is taken over the weights at $S$ connecting to $S'$ with $\phi(S')<\phi(S)$
	\end{enumerate}
	Moreover, $\{\tau_S, S\in \mathcal{S}\}$ give a $H^*_{T}(pt,\Q)$-additive basis of $H^*_{T}(G_k(\C^n),\Q)$.
\end{thm}

\begin{rmk}
	The canonical classes $\tau_S$ are exactly the equivariant Schubert classes via the relation that if $S$ consists of the elements $i_1<i_2<\cdots<i_k$, then the corresponding Schubert symbol is $(i_1-1,i_2-2,\ldots,i_k-k)$.
\end{rmk}

\begin{rmk}
	In this paper, we only need the existence of canonical classes $\tau_S$. The general formula of $\tau_S$ restricted at each fixed point was given by Guillemin\&Zara \cite{GZ03} and simplified by Goldin\&Tolman \cite{GT09}.
\end{rmk}

\subsection{Leray-Borel description of complex Grassmannians}
Besides the equivariant Schubert basis, equivariant characteristic classes and characteristic polynomials on complex Grassmannians will give ring generators and additive basis for their equivariant cohomology rings.

\begin{thm}[Leray-Borel description of complex Grassmannians, see \cite{BT82} pp.\,293 Prop\,23.2]
	For the complex Grassmannians $G_k(\C^n)$, let $c_1,c_2,\ldots,c_k$ and $\bar{c}_1,\bar{c}_2,\ldots,\bar{c}_{n-k}$ be the Chern classes of the canonical bundle on $G_k(\C^n)$ and its complementary bundle respectively, then
	\[
	H^*(G_k(\C^n),\Z)=\frac{\Z[c_1,c_2,\ldots,c_k;\bar{c}_1,\bar{c}_2,\ldots,\bar{c}_{n-k}]}{(1+c_1+c_2+\ldots+c_k)(1+\bar{c}_1+\bar{c}_2+\ldots+\bar{c}_{n-k})=1}.
	\]
\end{thm}

The relation $(1+c_1+c_2+\ldots+c_k)(1+\bar{c}_1+\bar{c}_2+\ldots+\bar{c}_{n-k})=1$ makes either $c_1,c_2,\ldots,c_k$ or $\bar{c}_1,\bar{c}_2,\ldots,\bar{c}_{n-k}$ as sets of ring generators of the cohomology $H^*(G_k(\C^n),\Z)$. Certain monomials of $c_1,c_2,\ldots,c_k$ actually give an additive basis of the cohomology $H^*(G_k(\C^n),\Z)$, stated by Carrell \cite{Ca78} for complex Grassmannians in $\C$ coefficients as a result of ``standard combinatorial reasoning''.  Later, details of proof were supplied by Jaworowski \cite{Ja89} for real Grassmannians in $\Z/2$ coefficients which can be adapted to complex Grassmannians in $\Z$ coefficients.

\begin{thm}[Characteristic basis of complex Grassmannians, \cite{Ca78, Ja89}]
	The set of monomials $c_1^{r_1}c_2^{r_2}\cdots c_k^{r_k}$ of cohomological degree $2d=\sum_{i=1}^{k} 2ir_i$ satisfying the condition $\sum_{i=1}^{k} r_i \leq n-k$ forms an additive basis for $H^{2d}(G_k(\C^n),\Z),0\leq d\leq k(n-k)$.
\end{thm}

Notice that the cohomology $H^*(G_k(\C^n),\Z)$ does not have odd-degree elements, the Leray-Serre sequence of $G_k(\C^n)\hookrightarrow ET\times_T G_k(\C^n) \rightarrow BT$ collapses at $E_2=H^*(BT,\Z)\otimes_\Z H^*(G_k(\C^n),\Z)$ with $H^*_T(G_k(\C^n),\Z)\cong H^*(BT,\Z)\otimes_\Z H^*(G_k(\C^n),\Z)$ as $H^*(BT,\Z)$-modules. Therefore, the action $T^n \curvearrowright G_k(\C^n)$ is equivariantly formal in $\Z$ coefficients, and of course in $\Q$ coefficients. However, $H^*_T(G_k(\C^n),\Z)\cong H^*(BT,\Z)\otimes_\Z H^*(G_k(\C^n),\Z)$ is only a $H^*(BT,\Z)$-module isomorphism which 
\begin{enumerate}
	\item neither gives the $H^*(BT,\Z)$-algebra structure of $H^*_T(G_k(\C^n),\Z)$.
	\item nor specifies a map $H^*_T(G_k(\C^n),\Z) \rightarrow H^*(BT,\Z)\otimes_\Z H^*(G_k(\C^n),\Z)$.
\end{enumerate}

The above two problems can be resolved using the equivariant version of the Leray-Borel description which is usually given for a compact connected Lie group $G$ with maximal torus $T$ and a closed connected subgroup $H$ containing $T$ as 
\[
H^*_T(G/H)=\mathbb{S}\mathfrak{t}^*\otimes_{(\mathbb{S}\mathfrak{t}^*)^{W_G}} (\mathbb{S}\mathfrak{t}^*)^{W_H}
\]
where $W_G$ and $W_H$ are the Weyl groups of $G$ and $H$. For the complex Grassmannian $G_k(\C^n)=U(n)/(U(k)\times U(n-k))$, we have the Weyl group $W_G = S_n$, the symmetric group of $n$ elements, and the Weyl group $W_H = S_k \times S_{n-k}$. Under these Weyl group actions, it is well known that the invariant elements in $\mathbb{S}\mathfrak{t}^*$ are symmetric polynomials, or topologically the equivariant Chern classes:

\begin{thm}[Equivariant Leray-Borel description of complex Grassmannians, \cite{Tu10} pp.\,21]
	For the complex Grassmannian $G_k(\C^n)=U(n)/(U(k)\times U(n-k))$, let $T^n$ be the maximal torus of $U(n)$ which acts on the left of $G_k(\C^n)$, and $\alpha_1,\alpha_2,\ldots,\alpha_n$ be the integral basis for its Lie dual algebra $\mathfrak{t}^*$, also let $c^T_1,c^T_2,\ldots,c^T_k$ and $\bar{c}^T_1,\bar{c}^T_2,\ldots,\bar{c}^T_{n-k}$ be the equivariant Chern classes of the canonical bundle on $G_k(\C^n)$ and its complementary bundle respectively, then
	\[
	H^*_T(G_k(\C^n),\Z)=\frac{\Z[\alpha_1,\alpha_2,\ldots,\alpha_n][c^T_1,c^T_2,\ldots,c^T_k;\bar{c}^T_1,\bar{c}^T_2,\ldots,\bar{c}^T_{n-k}]}{c^T\bar{c}^T = \prod_{i=1}^{n}(1+\alpha_i)}.
	\]
\end{thm}

Since the equivariant Chern classes $c^T_1,c^T_2,\ldots,c^T_k$ and $\bar{c}^T_1,\bar{c}^T_2,\ldots,\bar{c}^T_{n-k}$ lift the ordinary Chern classes $c_1,c_2,\ldots,c_k$ and $\bar{c}_1,\bar{c}_2,\ldots,\bar{c}_{n-k}$, we get

\begin{thm}[Equivariant characteristic basis of complex Grassmannians]
	The set of monomials $(c^T_1)^{r_1}(c^T_2)^{r_2}\cdots (c^T_k)^{r_k}$ satisfying the condition $\sum_{i=1}^{k} r_i \leq n-k$ form an additive $H^*_T(pt)$-basis for $H_T^*(G_k(\C^n),\Z)$.
\end{thm}

\begin{proof}
	Combine the ordinary characteristic basis with the equivalence (5) of Theorem \ref{thm:formal}.
\end{proof}

\subsection{Relations between the Leray-Borel and GKM descriptions}\label{subsec:BGKM}
Since the characteristic monomials $(c^T)^I=(c_1^T)^{i_1}\cdots(c_k^T)^{i_k},\,\sum_{j=1}^{k} i_j \leq n-k$ in Leray-Borel description and the canonical classes $\tau_{S}$ in GKM description are both basis for the free $\Q[\alpha_1,\ldots,\alpha_n]$-module $H^*_{T}(G_k(\C^n),\Q)$, there will be transformations $K,\bar{K}$ between them such that
\begin{align*}
(c^T)^I &= \sum_S K^I_S \tau_S\\
\tau_S &= \sum_I \bar{K}_I^S (c^T)^I
\end{align*}
where $K^I_S,\bar{K}_I^S \in \Q[\alpha_1,\ldots,\alpha_n]$.

\begin{rmk}
	The idea of considering the transformations between Schubert classes and characteristic classes dates back to Bernstein, Gelfand and Gelfand \cite{BGG73}. In this paper, we only need the existence of the transformations $K,\bar{K}$. For the complete flag manifold $Fl(\C^n)$, Kaji \cite{Ka} gave explicit algorithms on how to decide the polynomials $K^I_S,\bar{K}_I^S$. It would also be interesting to know what the $K^I_S,\bar{K}_I^S$ explicitly are for complex Grassmannians.
\end{rmk}

The Littlewood-Richardson rule for equivariant Schubert classes is that there are polynomials $N_{S,S'}^{S''}\in \Q[\alpha_1,\ldots,\alpha_n]$ (see Knutson\&Tao \cite{KT03}) for the multiplication of Schubert classes
\[
\tau_{S}\tau_{S'} =\sum_{S''}N_{S,S'}^{S''}\tau_{S''}.
\]

On the other hand, the multiplication of the equivariant characteristic classes can be localized. We can express the total equivariant Chern classes $c^T,\bar{c}^T$ of the canonical bundle $\gamma$ and its complementary bundle $\bar{\gamma}$ in GKM description at each fixed point $\oplus_{i\in S} \C_i \in G_k(\C^n)$. Note that the canonical bundle $\gamma$, complementary bundle $\bar{\gamma}$ and tangent bundle $TG_k(\C^n)$ restricted at $\oplus_{i\in S} \C_i \in G_k(\C^n)$ for a $k$-element subset $S \subset \{1,\ldots,n\}$ are the vector spaces $\oplus_{i\in S} \C_i$, $\oplus_{j\not\in S} \C_j $ and $\oplus_{i \in S} \oplus_{j \not \in S} (\C_i^* \otimes_\C \C_j)$ respectively, we get
\begin{align*}
c^T|_S &= c^T(\gamma|_S)=\prod_{i\in S} (1+\alpha_i)\\
\bar{c}^T|_S &= c^T(\bar{\gamma}|_S)=\prod_{j\not\in S} (1+\alpha_j)\\
e^T|_S &=e^T(T_SG_k(\C^n))=\prod_{i\in S}\prod_{j\not\in S} (\alpha_j - \alpha_i).
\end{align*}
Since $\gamma\oplus \bar{\gamma} = \sum_{i=1}^{n}\C_i$, this also shows why there is the relation $c^T\bar{c}^T = \prod_{i=1}^{n}(1+\alpha_i)$. 

If we denote $e_l(x_1,\ldots,x_m)$ as the $l$-th elementary symmetric polynomial in variables $x_1,\ldots,x_m$, then $c^T_l|_S = e_l(\alpha_{i\in S}), \bar{c}^T_l|_S = e_l(\alpha_{j\not\in S})$.
\begin{thm}[Equivariant Chern numbers of complex Grassmannians, \cite{Tu10} pp.\,21 Prop\,23]\label{thm:EquivChernNum}
	 Using the ABBV localization formula \ref{ABBV}, equivariant Chern numbers of complex Grassmannians can be given as
	\begin{align*}
	\int_{G_k(\C^n)} (c^T)^I 
	& = \sum_{S} \frac{((c_1^T)^{i_1}\cdots(c_k^T)^{i_k})|_S}{e^T_S} \\
	& = \sum_{S} \frac{e^{i_1}_1(\alpha_{i\in S})\cdots e^{i_k}_k(\alpha_{i\in S})}{\prod_{i\in S}\prod_{j\not\in S} (\alpha_j - \alpha_i)}\in \Q[\alpha_1,\ldots,\alpha_n].
	\end{align*}
	where the sum is taken for all $k$-element subsets $S \subset \{1,\ldots,n\}$. When the cohomological degree of a characteristic polynomial matches with the dimension of a Grassmannian, then its equivariant integration results in a constant, i.e. an ordinary Chern number. 
\end{thm}

\begin{rmk}
	By substituting any $\alpha_i=a_i \in \R$ such that $a_i\not =0, a_i \not = a_j$ into the above fraction, we can evaluate the equivariant Chern numbers. For example, a good choice will be $\alpha_i=i,\forall i$. If the cohomological degree of a characteristic polynomial does not match with the dimension of a Grassmannian, then such evaluation will be zero. The interesting case is when the degree matches the dimension.
\end{rmk}

\begin{cor}\label{thm:OrdChernNum}
	When the cohomological degree of a characteristic polynomial matches with the dimension of a Grassmannian, we then get a formula for the ordinary Chern numbers :
	\[
	\int_{G_k(\C^n)} c^I =  \sum_{S} \frac{e^{i_1}_1(S)\cdots e^{i_k}_k(S)}{\prod_{i\in S}\prod_{j\not\in S} (j - i)}\in \Q
	\]
	where the sum is taken for all $k$-element subsets $S \subset \{1,\ldots,n\}$.
\end{cor}

\begin{rmk}
	The above characteristic numbers are with respect to the characteristic classes of the canonical bundle and complementary bundle, not the tangent bundle. However,
	\[
	c^T(T(G_k(\C^n)))|_S=c^T(\oplus_{i\in S} \oplus_{j\not\in S} (\C_i^* \otimes \C_j))=\prod_{i\in S}\prod_{j\not\in S} (1+\alpha_j - \alpha_i).
	\]
	We can also use the ABBV formula to calculate the equivariant (and ordinary) characteristic numbers of the tangent bundle. 
\end{rmk}

\begin{rmk}
	The equivariant Chern classes $(c^T)^I$ are integral. Moreover, the canonical classes $\tau_S$ are actually the equivariant Schubert classes, hence also integral. Therefore, the coefficients $N,K,\bar{K}$ and characteristic numbers $\int (c^T)^I,\int c^I$ are all integral.
\end{rmk}

\vskip 20pt
\section{Equivariant cohomology rings of real Grassmannians}
\vskip 15pt
In this section, we give the GKM description and Leray-Borel description of equivariant cohomology rings of real Grassmannians, together with the canonical basis and characteristic basis of the additive structure. We use the notation $G_k(\R^n)$ for the Grassmannian of $k$-dimensional real subspaces in $\R^n$.

The dimension of $G_k(\R^n)$ is $k(n-k)$. Therefore, $G_{2k}(\R^{2n}), G_{2k}(\R^{2n+1}), G_{2k+1}(\R^{2n+1})$ are even dimensional, but $G_{2k+1}(\R^{2n+2})$ is odd dimensional. Moreover, the real Grassmannians $G_k(\R^n)$ differ from each other on Poincar\'e series and orientability according to the parities of $k$ and $n$, as shown by Casian and Kodama:

\begin{thm}[Poincar\'e series of real Grassmannians, \cite{CK} pp.\,11, Thm\,5.1]\label{thm:Poinc}
	The relations between Poincar\'e series of real Grassmannians and complex Grassmannians are given as:
	\begin{align*}
	P_{G_{2k}(\R^{2n})}(t)=P_{G_{2k}(\R^{2n+1})}(t)&=P_{G_{2k+1}(\R^{2n+1})}(t)=P_{G_{k}(\C^{n})}(t^2)\\
	P_{G_{2k+1}(\R^{2n+2})}(t)&=(1+t^{2n+1})P_{G_{k}(\C^{n})}(t^2).
	\end{align*}
\end{thm}

\begin{rmk}
	The Poincar\'e series of complex Grassmannian is (see \cite{BT82} pp.\,292 Prop\,23.1)
	\[
	P_{G_{k}(\C^{n})}(t)=\frac{(1-t^2)\cdots(1-t^{2n})}{(1-t^2)\cdots(1-t^{2k})(1-t^2)\cdots(1-t^{2(n-k)})}.
	\]
	Using the relations between Poincar\'e series, we see that $G_{2k}(\R^{2n})$ and $G_{2k+1}(\R^{2n+2})$ have non-zero top Betti numbers, hence orientable; however, $G_{2k}(\R^{2n+1})$ and $G_{2k+1}(\R^{2n+1})$ have zero top Betti numbers, hence non-orientable.
\end{rmk}

\subsection{GKM description of real Grassmannians}
Similar to the case of complex Grassmannians, we will show the real Grassmannians also have appropriate torus actions that are equivariantly formal and GKM. 

First, we specify the torus actions on real Grassmannians. Write the coordinates on $\R^{2n}$ as $(x_1,y_1,\ldots,x_n,y_n)$. Let $T^n$ act on $\R^{2n},\R^{2n+1},\R^{2n+2}$ so that the $i$-th $S^1$-component of $T^n$ exactly rotates the $i$-th pair of real coordinates $(x_{i},y_{i})$ and leaves the remaining coordinates free, hence we can write $\R^{2n}=\oplus_{i=1}^{n} \R^2_{[\alpha_i]}$, $\R^{2n+1}=(\oplus_{i=1}^{n} \R^2_{[\alpha_i]})\oplus \R_0$ and $\R^{2n+2}=(\oplus_{i=1}^{n} \R^2_{[\alpha_i]})\oplus \R^2_0$ for their decompositions into weighted subspaces, where $[\alpha_i] \in \mathfrak{t}_\Z^*/\pm 1$. These actions induce $T^n$ actions on $G_{2k}(\R^{2n})$, $G_{2k}(\R^{2n+1})$, $G_{2k+1}(\R^{2n+1})$ and $G_{2k+1}(\R^{2n+2})$.

Since there are natural $T^n$-diffeomorphisms $G_{2k}(\R^{2n+1})\cong G_{2n-2k+1}(\R^{2n+1})$ identifying the second and the third types of real Grassmannians, in many discussions we will only consider the three cases of $G_{2k}(\R^{2n})$, $G_{2k}(\R^{2n+1})$ and $G_{2k+1}(\R^{2n+2})$.

\subsubsection{Fixed points}
Similar to the observation in the case of complex Grassmannians, the $T^n$-fixed points of real Grassmannians are exactly some appropriate dimensional sub-representations of the ambient representations. The verification of sub-representations of $\R^{2n}=\oplus_{i=1}^{n} \R^2_{[\alpha_i]}$, $\R^{2n+1}=(\oplus_{i=1}^{n} \R^2_{[\alpha_i]})\oplus \R_0$ for the real Grassmannians $G_{2k}(\R^{2n})$ and $G_{2k}(\R^{2n+1})$ are straightforward. Let's focus on the $2k+1$ dimensional sub-representations of $\R^{2n+2}=(\oplus_{i=1}^{n} \R^2_{[\alpha_i]})\oplus \R^2_0$. Notice that the sub-representation is odd dimensional, hence must have exactly one dimension in the part of trivial representation, therefore has the form $(\oplus_{i\in S} \R^2_{[\alpha_i]})\oplus L_0$ where $S$ is a $k$-element subset of $\{1,2,\ldots,n\}$ and $L_0\in \P(\R^2_0)$. For each $k$-element subset $S$, the connected component $C_S=\{(\oplus_{i\in S} \R^2_{[\alpha_i]})\oplus L_0 \mid L_0\in \P(\R^2_0)\}\cong \R P^1$ gives a fixed circle isolated from the other fixed circles. This gives all the fixed points of the $T^n$ action on $G_{2k+1}(\R^{2n+2})$.

\begin{rmk}
	Because of the one-to-one correspondence between a $k$-element subset $S \subset \{1,\ldots,n\}$ with a fixed point or circle, sometimes we will use $S$ directly to mean a fixed point or circle.
\end{rmk}

\subsubsection{Isotropy weights}\label{subsubsec:IsoWeights}
Fixing a $k$-element subset $S$, let's describe the tangent spaces at the fixed points in the three cases of real Grassmannians.
\begin{enumerate}
	\item The tangent space at $\oplus_{i\in S} \R^2_{[\alpha_i]} \in G_{2k}(\R^{2n})$ is 
	\[
	Hom_\R\Big(\oplus_{i\in S} \R^2_{[\alpha_i]}, \oplus_{j \not\in S} \R^2_{[\alpha_j]}\Big) \cong (\oplus_{i\in S} \R^2_{[\alpha_i]})^* \otimes_\R (\oplus_{j \not\in S} \R^2_{[\alpha_j]})\cong \oplus_{i \in S} \oplus_{j \not \in S} \big((\R^2_{[\alpha_i]})^* \otimes_\R \R^2_{[\alpha_j]}\big).
	\]
	\item The tangent space at $\oplus_{i\in S} \R^2_{[\alpha_i]} \in G_{2k}(\R^{2n+1})$ is 
	\[
	Hom_\R\Big(\oplus_{i\in S} \R^2_{[\alpha_i]}, (\oplus_{j \not\in S} \R^2_{[\alpha_j]})\oplus \R_0\Big) \cong \Big(\oplus_{i \in S} \oplus_{j \not \in S} \big((\R^2_{[\alpha_i]})^* \otimes_\R \R^2_{[\alpha_j]}\big)\Big)\oplus \oplus_{i \in S} (\R^2_{[\alpha_i]})^*.
	\]
	\item The tangent space at $(\oplus_{i\in S} \R^2_{[\alpha_i]})\oplus L_0 \in G_{2k+1}(\R^{2n+2})$, where $L_0 \in \P(\R^2_0)$ has a $L_0^\bot \in \P(\R^2_0)$ such that $L_0 \oplus L_0^\bot \cong \R^2_0$, is 
	\begin{align*}
	&Hom_\R\Big((\oplus_{i\in S} \R^2_{[\alpha_i]})\oplus L_0, (\oplus_{j \not\in S} \R^2_{[\alpha_j]})\oplus L_0^\bot\Big)\\ 
	\cong& \Big(\oplus_{i \in S} \oplus_{j \not \in S} \big((\R^2_{[\alpha_i]})^* \otimes_\R \R^2_{[\alpha_j]}\big)\Big) \oplus \Big( \oplus_{i \in S} (\R^2_{[\alpha_i]})^*\otimes_\R L_0^\bot  \Big) \oplus \Big( \oplus_{j \not\in S} L_0^*\otimes_\R \R^2_{[\alpha_j]} \Big) \\ 
	&\oplus\Big( L_0^* \otimes_\R L_0^\bot \Big) 
	\end{align*}
	among which the first three terms and the fourth term give respectively the normal space and tangent space of the fixed circle $C_S=\{(\oplus_{i\in S} \R^2_{[\alpha_i]})\oplus L_0 \mid L_0\in \P(\R^2_0)\}$.
\end{enumerate}
The isotropy weights are then determined by the following simple lemma:
\begin{lem}
	The weights of the tensor product $(\R^2_{[\alpha_i]})^* \otimes_\R \R^2_{[\alpha_j]}$ are $[\alpha_j-\alpha_i],[\alpha_j+\alpha_i] \in \mathfrak{t}_\Z^*/\pm 1$.
\end{lem}
\begin{proof}
	The $T$-action on the dual space $(\R^2_{[\alpha_i]})^*$ is defined in an invariant way so that for $t \in T$, $l \in (\R^2_{[\alpha_i]})^*$, $v \in \R^2_{[\alpha_i]}$, and if we denote $\langle l,v\rangle$ as the natural pairing, we should have
	$\langle t\cdot l,t\cdot v\rangle=\langle l,v\rangle$ or equivalently, $(t\cdot l)(v) = l(t^{-1}\cdot v)$. Notice that only the $i$-th and $j$-th $S^1$-component of $T^n$ have non-trivial actions on $(\R^2_{[\alpha_i]})^*$ or $\R^2_{[\alpha_j]}$, let $e^{\sqrt{-1}\theta_i} \in S^1_i$, $e^{\sqrt{-1}\theta_j} \in S^1_j$, and write elements of $(\R^2_{[\alpha_i]})^* \otimes_\R \R^2_{[\alpha_j]}$ as $2 \times 2$ matrices, then the $S^1_i \times S^1_j$ action on $(\R^2_{[\alpha_i]})^* \otimes_\R \R^2_{[\alpha_j]}$ can be given as:
	\[
	(e^{\sqrt{-1}\theta_i},e^{\sqrt{-1}\theta_j})\cdot 
	\begin{pmatrix}
	a & b\\
	c & d
	\end{pmatrix}=
	\begin{pmatrix}
	\cos \theta_j & -\sin \theta_j\\
	\sin \theta_j & \cos \theta_j
	\end{pmatrix}
	\begin{pmatrix}
	a & b\\
	c & d
	\end{pmatrix}
	\begin{pmatrix}
	\cos \theta_i & \sin \theta_i\\
	-\sin \theta_i & \cos \theta_i
	\end{pmatrix}.	
	\]
	Consider the following new basis of $(\R^2_{[\alpha_i]})^* \otimes_\R \R^2_{[\alpha_j]}$
	\[
	M_1 = 
	\begin{pmatrix}
	1 & 0\\
	0 & 1
	\end{pmatrix} \quad
	M_2 = 
	\begin{pmatrix}
	0 & -1\\
	1 & 0
	\end{pmatrix} \quad
	M_3 = 
	\begin{pmatrix}
	1 & 0\\
	0 & -1
	\end{pmatrix} \quad
	M_4 = 
	\begin{pmatrix}
	0 & 1\\
	1 & 0
	\end{pmatrix}.	
	\]
	We get 
	\begin{align*}
	(e^{\sqrt{-1}\theta_j}\cdot M_1,e^{\sqrt{-1}\theta_j}\cdot M_2) &= (M_1,M_2)
	\begin{pmatrix}
	\cos \theta_j & -\sin \theta_j\\
	\sin \theta_j & \cos \theta_j
	\end{pmatrix}\\
	(e^{\sqrt{-1}\theta_i}\cdot M_1,e^{\sqrt{-1}\theta_i}\cdot M_2) &= (M_1,M_2)
	\begin{pmatrix}
	\cos \theta_i & \sin \theta_i\\
	-\sin \theta_i & \cos \theta_i
	\end{pmatrix}.
	\end{align*}
	In other words, $S^1_i \times S^1_j$ acts on $\R M_1 \oplus \R M_2$ with weight $[\alpha_j-\alpha_i]$.
	Similarly, $S^1_i \times S^1_j$ acts on $\R M_3 \oplus \R M_4$ with weight $[\alpha_j+\alpha_i]$.
\end{proof}

\subsubsection{$1$-skeleta}
To begin with, let's work out the $1$-skeleton of the $T^2$-action on $G_2(\R^4)$. From the previous discussions, we know that there are two fixed points $\R^2_{[\alpha_1]},\R^2_{[\alpha_2]} \in G_2(\R^4)$, both have the same isotropy weights $[\alpha_2-\alpha_1]$ and $[\alpha_2+\alpha_1]$. Let $T_{\alpha_2-\alpha_1}$ be the subtorus of $T^2$ with Lie algebra annihilated by $\alpha_2-\alpha_1$, i.e. $T_{\alpha_2-\alpha_1}$ is the diagonal $\{(t,t) \in T^2\}$. Similarly, $T_{\alpha_2+\alpha_1}$, the subtorus with Lie algebra annihilated by $\alpha_2+\alpha_1$, is the anti-diagonal $\{(t,t^{-1}) \in T^2\}$.

Note that there is a natural diffeomorphism $\mathcal{F}:\C^2 \rightarrow \R^4$ by forgetting the complex structure. This induces an embedding $\C P^1 \hookrightarrow G_2(\R^4):L \mapsto \mathcal{F}(L)$ where $L$ is a complex line in $\C^2$ and $\mathcal{F}(L)$ its two dimensional real image in $\R^4$. Let $J:\C^2 \rightarrow \C^2: (z_1,z_2) \mapsto (z_1,\bar{z}_2)$ be the diffeomorphism with conjugation on the second variable. This also induces an embedding $\C P^1 \hookrightarrow G_2(\R^4):L \mapsto \mathcal{F}(J(L))$. We will denote the images of the two embeddings as $\C P^1$ and $\overline{\C P^1}$.

\begin{lem}
	The fixed-point sets of $T_{\alpha_2-\alpha_1}$ and $T_{\alpha_2+\alpha_1}$ in $G_2(\R^4)$ are $\C P^1$ and $\overline{\C P^1}$ respectively, i.e. the $1$-skeleton of the $T^2$-action on $G_2(\R^4)$ is $\C P^1 \cup \overline{\C P^1}$ glued at the two $T^2$-fixed points $\R^2_{[\alpha_1]},\R^2_{[\alpha_2]} \in G_2(\R^4)$.
\end{lem}
\begin{proof}
	Let $L_0=\C \oplus 0$ and $L_\infty = 0 \oplus \C$ be the two complex lines in $\C^2$, they are the two poles of both $\C P^1$ and $\overline{\C P^1}$, and are exactly the two $T^2$-fixed points $\R^2_{[\alpha_1]},\R^2_{[\alpha_2]} \in G_2(\R^4)$. The diagonal circle $T_{\alpha_2-\alpha_1}=\{(t,t) \in T^2\}$ fixes $\C P^1$ because $(t,t)\cdot [z_1,z_2]=[tz_1,tz_2]=[z_1,z_2]$ trivially, hence $\C P^1$ joins $\R^2_{[\alpha_1]}$ to $\R^2_{[\alpha_2]}$ with weight $[\alpha_2-\alpha_1]$. Similarly, $\overline{\C P^1}$ joins $\R^2_{[\alpha_1]}$ to $\R^2_{[\alpha_2]}$ with weight $[\alpha_2+\alpha_1]$. The $2$-spheres $\C P^1$ and $\overline{\C P^1}$ exhaust all the $T^2$-fixed points and the isotropy weights, therefore give the $1$-skeleton of the $T^2$-action on $G_2(\R^4)$.
\end{proof}

Generally, let $T_{\alpha_j-\alpha_i}$ and $T_{\alpha_j+\alpha_i}$ be the subtori of $T^n$ with Lie algebras annihilated by $\alpha_j-\alpha_i$ and $\alpha_j+\alpha_i$ respectively. For the $T^n$-action on $\R^{2n}=\oplus_{i=1}^{n} \R^2_{[\alpha_i]}$, the fixed-point sets of $T_{\alpha_j-\alpha_i}$ and $T_{\alpha_j+\alpha_i}$ on $G_2(\R^2_{[\alpha_i]}\oplus \R^2_{[\alpha_j]})$ are two $2$-spheres sharing the poles which are exactly the two $T^n$-fixed points $\R^2_{[\alpha_i]},\R^2_{[\alpha_j]} \in G_2(\R^2_{[\alpha_i]}\oplus \R^2_{[\alpha_j]})$. We will denote the $2$-spheres as $S^2_{[\alpha_j-\alpha_i]}$ and $S^2_{[\alpha_j+\alpha_i]}$ and keep in mind that every element $V$ in $S^2_{[\alpha_j-\alpha_i]}$ or $S^2_{[\alpha_j+\alpha_i]}$ is a $2$-plane in $\R^2_{[\alpha_i]}\oplus \R^2_{[\alpha_j]}$.

Now we are ready to describe the $1$-skeleta of the $T^n$ actions on the three types of real Grassmannians. Let $S$ be a $k$-element subset of $\{1,2,\ldots,n\}$, and $i\in S$, $j\not\in S$.
\begin{enumerate}
	\item For $G_{2k}(\R^{2n})$, the $T^n$-fixed point $V_S=\oplus_{i'\in S} \R^2_{[\alpha_{i'}]}$ is joined to $V_{(S\setminus\{i\})\cup \{j\}}$ via $\{(\oplus_{i'\in S\setminus\{i\}} \R^2_{[\alpha_{i'}]})\oplus V\mid V\in S^2_{[\alpha_j-\alpha_i]}\}\cong S^2$ of weight $[\alpha_j-\alpha_i]$ and also via $\{(\oplus_{i'\in S\setminus\{i\}} \R^2_{[\alpha_{i'}]})\oplus V\mid V\in S^2_{[\alpha_j+\alpha_i]}\}\cong S^2$ of weight $[\alpha_j+\alpha_i]$.
	\item For $G_{2k}(\R^{2n+1})$, the $T^n$-fixed point $V_S=\oplus_{i'\in S} \R^2_{[\alpha_{i'}]}$ is joined to $V_{(S\setminus\{i\})\cup \{j\}}$ via $\{(\oplus_{i'\in S\setminus\{i\}} \R^2_{[\alpha_{i'}]})\oplus V\mid V\in S^2_{[\alpha_j-\alpha_i]}\}\cong S^2$ of weight $[\alpha_j-\alpha_i]$ and also via $\{(\oplus_{i'\in S\setminus\{i\}} \R^2_{[\alpha_{i'}]})\oplus V\mid V\in S^2_{[\alpha_j+\alpha_i]}\}\cong S^2$ of weight $[\alpha_j+\alpha_i]$. Moreover, $V_S$ is contained in $\{(\oplus_{i'\in S\setminus\{i\}} \R^2_{[\alpha_{i'}]})\oplus V\mid V\in G_2(\R^2_{[\alpha_i]}\oplus \R_0)\}\cong \R P^2$ of weight $[\alpha_i]$ without other fixed points.
	\item For $G_{2k+1}(\R^{2n+2})$, the $T^n$-fixed circle $C_S=\{(\oplus_{i'\in S} \R^2_{[\alpha_{i'}]})\oplus L_0 \mid L_0\in \P(\R^2_0)\}\cong \R P^1$ is joined to $C_{(S\setminus\{i\})\cup \{j\}}$ via $\{(\oplus_{i'\in S\setminus\{i\}} \R^2_{[\alpha_{i'}]})\oplus V \oplus L_0 \mid V\in S^2_{[\alpha_j-\alpha_i]}, L_0\in \P(\R^2_0)\}\cong S^2\times \R P^1$ with weight $[\alpha_j-\alpha_i]$ and also via $\{(\oplus_{i'\in S\setminus\{i\}} \R^2_{[\alpha_{i'}]})\oplus V \oplus L_0 \mid V\in S^2_{[\alpha_j+\alpha_i]}, L_0\in \P(\R^2_0)\}\cong S^2\times \R P^1$ with weight $[\alpha_j+\alpha_i]$. Moreover, $C_S$ is contained in $\{(\oplus_{i'\in S\setminus\{i\}} \R^2_{[\alpha_{i'}]})\oplus W \mid W\in G_3(\R^2_{[\alpha_{i}]}\oplus \R^2_0)\}\cong \R P^3$ and $\{(\oplus_{i'\in S} \R^2_{[\alpha_{i'}]})\oplus L  \mid L\in \P(\R^2_{[\alpha_j]}\oplus\R^2_0)\}\cong \R P^3$ of weights $[\alpha_i]$ and $[\alpha_j]$ respectively without other fixed points. 	
\end{enumerate}

\subsubsection{GKM graphs of real Grassmannians}
Since $G_{2k}(\R^{2n}), G_{2k}(\R^{2n+1}), G_{2k+1}(\R^{2n+1})$ are even dimensional, but $G_{2k+1}(\R^{2n+2})$ is odd dimensional, we will construct GKM graphs according to the parity of dimensions.

\begin{exm}
	We give some examples of GKM graphs for $G_k(\R^n)$ when $k$ or $n$ is small.
	\begin{enumerate}
		\item $\R P^{2n}$ as $G_{1}(\R^{2n+1})$ or $G_{2n}(\R^{2n+1})$ 
		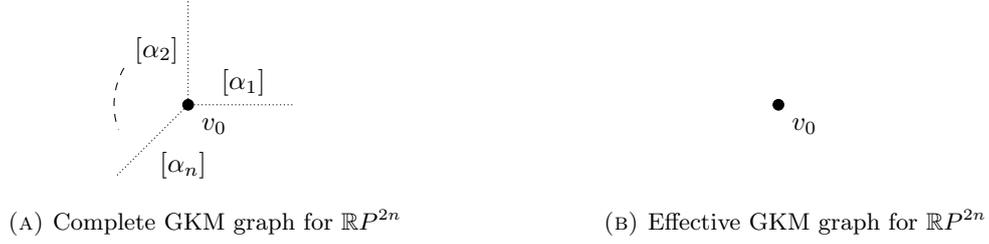
\begin{figure}[H]
			\centering
			\begin{subfigure}[b]{.45\textwidth}
				\centering
				\includegraphics{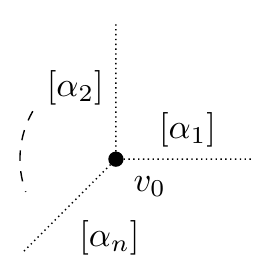}
				\caption{Complete GKM graph for $\R P^{2n}$}
			\end{subfigure}
			\begin{subfigure}[b]{.45\textwidth}
				\centering
				\begin{tikzpicture}[scale=1.5]
				\path (0:0cm) node[draw,fill,circle,inner sep=0pt,minimum size=4pt,label=below right:{$v_0$}] (v0) {};
				\path (0:1cm) node (v1) {}
				(90:1cm) node (v2) {}
				(225:1cm) node (v3) {};
				\end{tikzpicture}
				\caption{Effective GKM graph for $\R P^{2n}$}
			\end{subfigure}
			\caption{GKM graphs for $\R P^{2n}$}
		\end{figure}
		\item $\R P^{2n+1}$ as $G_{1}(\R^{2n+2})$ or $G_{2n+1}(\R^{2n+2})$ 
		\begin{figure}[H]
			\centering
			\begin{subfigure}[b]{.45\textwidth}
				\centering
				\includegraphics{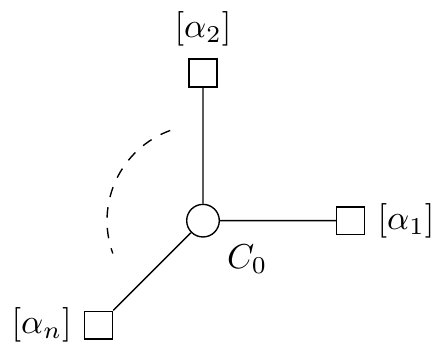}
				\caption{GKM graph for $\R P^{2n+1}$}
			\end{subfigure}
			\begin{subfigure}[b]{.45\textwidth}
				\centering
				\includegraphics{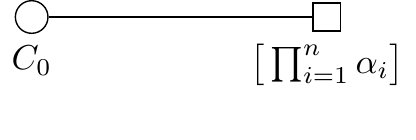}
				\caption{Condensed GKM graph for $\R P^{2n+1}$}
			\end{subfigure}
			\caption{GKM graphs for $\R P^{2n+1}$}
		\end{figure}
		\item $G_{2}(\R^{4}),G_{2}(\R^{5}),G_{3}(\R^{5})$ as $G_{2k}(\R^{2n}), G_{2k}(\R^{2n+1}), G_{2k+1}(\R^{2n+1})$ when $k=1,n=2$.
		\begin{figure}[H]
			\centering
			\begin{subfigure}[b]{0.45\textwidth}
				\centering
				\includegraphics{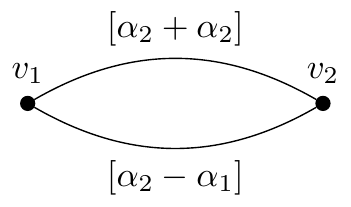}
				\caption{GKM graph for $G_{2}(\R^{4})$}
			\end{subfigure}
			\begin{subfigure}[b]{0.45\textwidth}
				\centering
				\includegraphics{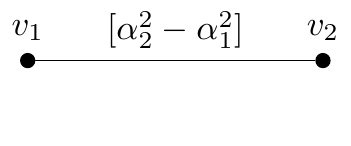}
				\caption{Condensed GKM graph for $G_{2}(\R^{4})$}
			\end{subfigure}
			\caption{GKM graphs for $G_{2}(\R^{4})$}
		\end{figure}
		\begin{figure}[H]
			\centering
			\begin{subfigure}[b]{0.3\textwidth}
				\centering
				\includegraphics{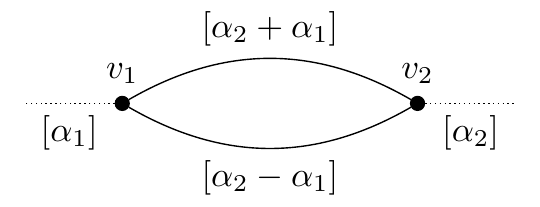}
				\caption{Complete GKM graph}
			\end{subfigure}
			\begin{subfigure}[b]{0.3\textwidth}
				\centering
				\includegraphics{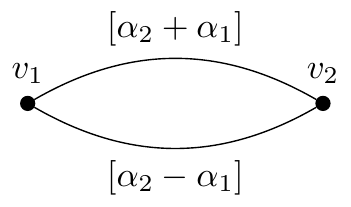}
				\caption{Effective GKM graph}
			\end{subfigure}
			\begin{subfigure}[b]{0.3\textwidth}
				\centering
				\includegraphics{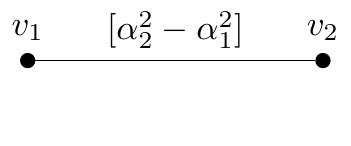}
				\caption{Condensed GKM graph}
			\end{subfigure}
			\caption{GKM graphs for $G_{2}(\R^{5})$}
		\end{figure}
		\begin{figure}[H]
			\centering
			\begin{subfigure}[b]{0.3\textwidth}
				\centering
				\includegraphics{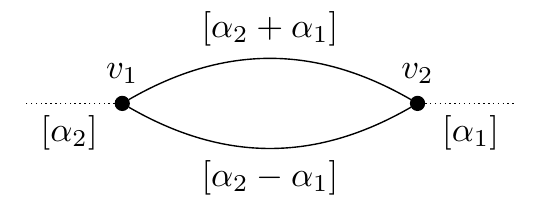}
				\caption{Complete GKM graph}
			\end{subfigure}
			\begin{subfigure}[b]{0.3\textwidth}
				\centering
				\includegraphics{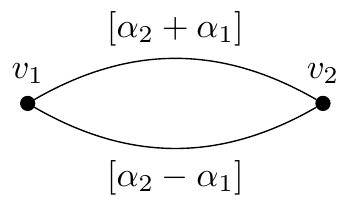}
				\caption{Effective GKM graph}
			\end{subfigure}
			\begin{subfigure}[b]{0.3\textwidth}
				\centering
				\includegraphics{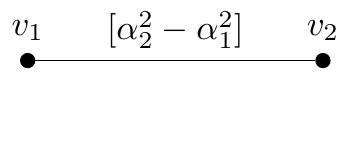}
				\caption{Condensed GKM graph}
			\end{subfigure}
			\caption{GKM graphs for $G_{3}(\R^{5})$}
		\end{figure}		
		\item $G_{3}(\R^{6})$ as $G_{2k+1}(\R^{2n+2})$ when $k=1,n=2$.
		\begin{figure}[H]
			\centering
			\begin{subfigure}[b]{0.45\textwidth}
				\centering
				\includegraphics{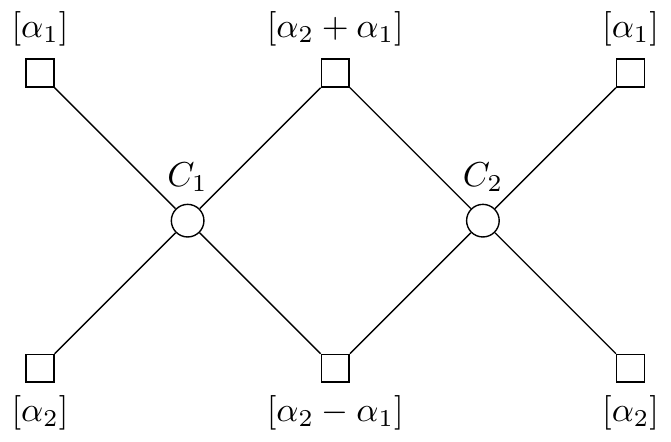}
				\caption{GKM graph for $G_{3}(\R^{6})$}
			\end{subfigure}
			\begin{subfigure}[b]{0.45\textwidth}
				\centering
				\includegraphics{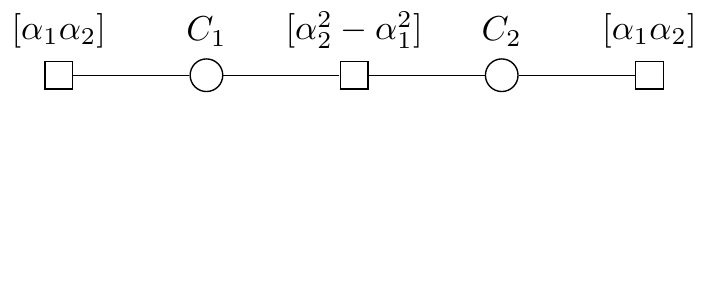}
				\caption{Condensed GKM graph for $G_{3}(\R^{6})$}
			\end{subfigure}
			\caption{GKM graphs for $G_{3}(\R^{6})$}
		\end{figure}
	\end{enumerate}
\end{exm}

\begin{rmk}
	The graphs of $\R P^{2n}$ and $G_2(\R^5)$ have appeared in Goertsches\&Mare \cite{GM14}.
\end{rmk}

\subsubsection{Formality, cohomology and canonical basis of real Grassmannians}
We have given the $1$-skeleta and GKM graphs for real Grassmannians $G_k(\R^n)$ under appropriate torus actions. To apply the GKM-type theorems in even and odd dimensions, we still need to verify that those torus actions on $G_k(\R^n)$ are equivariantly formal. 

\begin{prop}[Equivariant formality of torus actions on real Grassmannians]
	The total Betti numbers of $G_k(\R^n)$ and of its fixed-point set are equal:
	\begin{enumerate}
		\item For the $T^n$-actions on $G_{2k}(\R^{2n}), G_{2k}(\R^{2n+1}), G_{2k+1}(\R^{2n+1})$, the isolated fixed points $V_S$ are all parametrized by $\mathcal{S}=\{S\subseteq\{1,2,\ldots,n\} \mid \#S=k\}$, and we have
		\[
		\sum \mathrm{dim}\,H^*(G_{2k}(\R^{2n})) = \sum \mathrm{dim}\,H^*(G_{2k}(\R^{2n+1})) = \sum \mathrm{dim}\,H^*(G_{2k+1}(\R^{2n+1})) = \# \mathcal{S} = \binom{n}{k}.
		\]
		\item For the $T^n$-action on $G_{2k+1}(\R^{2n+2})$, the isolated fixed circles $C_S$ are also indexed on $\mathcal{S}=\{S\subseteq\{1,2,\ldots,n\} \mid \#S=k\}$, and we have
		\[
		\sum\mathrm{dim}\,H^*(G_{2k+1}(\R^{2n+1})) = \# \mathcal{S} \cdot \sum\mathrm{dim}\,H^*(S^1) = 2\binom{n}{k}.
		\]
	\end{enumerate}
	Therefore, the torus actions on $G_k(\R^n)$ are equivariantly formal.
\end{prop}
\begin{proof}
	The verification is based on the equivalence (6) of Theorem \ref{thm:formal}. The total Betti numbers of $G_k(\R^n)$ can be calculated from the Casian-Kodama formula in Theorem \ref{thm:Poinc} by substituting $t=1$ in the Poincar\'e series.
	\begin{align*}
	\sum \mathrm{dim}\,H^*(G_{2k}(\R^{2n})) = \sum \mathrm{dim}\,H^*(G_{2k}(\R^{2n+1})) &= \sum \mathrm{dim}\,H^*(G_{2k+1}(\R^{2n+1})) = \sum \mathrm{dim}\,H^*(G_{k}(\C^{n}))\\
    \sum \mathrm{dim}\,H^*(G_{2k+1}(\R^{2n+2})) &= 2\sum \mathrm{dim}\,H^*(G_{k}(\C^{n})).
	\end{align*}
	On the other hand, by the formality of the $T^n$-action on $G_{k}(\C^{n})$, which also has isolated points parametrized by $\mathcal{S}$, we have 
	\[
	\sum \mathrm{dim}\,H^*(G_{k}(\C^{n})) = \# \mathcal{S} = \binom{n}{k}.
	\]
	Therefore, total Betti numbers of $G_k(\R^n)$ and of its fixed-point set are equal, and the torus actions on $G_k(\R^n)$ are equivariantly formal.
\end{proof}

With the verifications of GKM conditions and equivariant formality, we can give the GKM description of the torus actions on $G_k(\R^n)$ by applying the generalized GKM-type Theorems \ref{thm:EvenGKM} and \ref{thm:OddGKM} in even and odd dimensions. 

\begin{thm}[GKM description of equivariant cohomology of real Grassmannians]\label{thm:GKMrealGrass}
	Let $\mathcal{S}$ be the collection of $k$-element subsets of $\{1,2,\ldots,n\}$. 
	\begin{enumerate}
		\item For even dimensional Grassmannians $G_{2k}(\R^{2n}), G_{2k}(\R^{2n+1}), G_{2k+1}(\R^{2n+1})$ with $T^n$-actions, they have the same equivariant cohomology
		\[
		\big\{f:\mathcal{S}\rightarrow \Q[\alpha_1,\ldots,\alpha_n] \mid f_{S} \equiv f_{S'} \mod \alpha^2_j-\alpha^2_i \quad \text{for $S,S' \in \mathcal{S}$ with $S\cup\{j\}=S'\cup\{i\}$}\big\}.
		\]
		\item For odd dimensional Grassmannian $G_{2k+1}(\R^{2n+2})$ with $T^n$-action, an element of the equivariant cohomology is a set of polynomial pairs $(f_S, g_S \theta)$ to each $\circ$-vertex $S$ where $\theta$ is the unit volume form of $S^1$ such that
		\begin{enumerate}
			\item $g_{S} \equiv 0 \mod \prod_{i=1}^{n}\alpha_i$\quad for every $S$
			\item $f_{S} \equiv f_{S'} , \quad g_{S} \equiv g_{S'}\mod \alpha^2_j-\alpha^2_i$\quad for $S,S' \in \mathcal{S}$ with $S\cup\{j\}=S'\cup\{i\}$.
		\end{enumerate}
	\end{enumerate}
\end{thm}

\begin{rmk}
	For convenience, we will write an element $f \in H^*_{T^n}(G_{2k}(\R^{2n}))$ as $(f_S)_{S \in \mathcal{S}}$ and an element $(f,g\theta) \in H^*_{T^n}(G_{2k+1}(\R^{2n+2}))$ as $(f_S+g_S\theta)_{S \in \mathcal{S}}$, which are understood as tuples indexed with respect to $S \in \mathcal{S}$.
\end{rmk}

\begin{rmk}
	In the $1$-skeleton of the odd dimensional Grassmannian $G_{2k+1}(\R^{2n+2})$, every $\R P^3_{[\alpha_i]}$ containing a unique fixed circle $C_S$ contributes a relation $g_{S} \equiv 0 \mod \alpha_i$; every $S^2\times \R P^1$ with weight $\alpha_j\pm \alpha_i$ and two fixed circles $C_S,C_{S'}$ contributes two relations $f_{S} \equiv f_{S'} ,  g_{S} \equiv g_{S'}\mod \alpha_j\pm\alpha_i$. These simple components in $1$-skeleton resolve the sign issues in odd dimensional GKM-type Theorem\,\ref{thm:OddGKM}.
\end{rmk}

\begin{rmk}
	Note that in the above description, we have condensed some congruence relations because $\Q[\alpha_1,\ldots,\alpha_n]$ is a unique-factorization domain.
	\begin{align*}
		\begin{cases}
		f_{S} \equiv f_{S'} \mod \alpha_j-\alpha_i\\
		f_{S} \equiv f_{S'} \mod \alpha_j+\alpha_i
		\end{cases}
		&\Longleftrightarrow \quad f_{S} \equiv f_{S'} \mod \alpha^2_j-\alpha^2_i\\
	    \begin{cases}
		g_{S} \equiv 0 \mod \alpha_1\\
		\vdots \quad \qquad \vdots \quad \qquad \vdots\\
		g_{S} \equiv 0 \mod \alpha_n
		\end{cases} 
		&\Longleftrightarrow \quad g_{S} \equiv 0 \mod \prod_{i=1}^{n}\alpha_i.
	\end{align*}
\end{rmk}

Notice the similarity among the GKM descriptions of the even and odd dimensional real Grassmannians and the complex Grassmannians, we have
\begin{thm}[Relations among equivariant cohomology of real and complex Grassmannians]\label{thm:AllGrass}
	The relations between the equivariant cohomology of even, odd dimensional real Grassmannians and complex Grassmannians are
	\begin{enumerate}
		\item There are a series of $\Q[\alpha_1,\ldots,\alpha_n]$-algebra isomorphisms:
		\[
		H^*_{T^n}(G_{2k}(\R^{2n}))\cong H^*_{T^n}(G_{2k}(\R^{2n+1}))\cong H^*_{T^n}(G_{2k+1}(\R^{2n+1})).
		\]
		\item There is an element $r^T \in H^{2n+1}_{T^n}(G_{2k+1}(\R^{2n+2}))$ such that $(r^T)^2=0$, and there is a $\Q[\alpha_1,\ldots,\alpha_n]$-algebra isomorphism
		\[
		H^*_{T^n}(G_{2k+1}(\R^{2n+2})) \cong H^*_{T^n}(G_{2k}(\R^{2n}))[r^T]/(r^T)^2.
		\]
		\item There is a $\Q[\alpha_1,\ldots,\alpha_n]$-algebra monomorphism:
		\[
		H^*_{T^n}(G_{2k}(\R^{2n})) \hookrightarrow H^*_{T^n}(G_{k}(\C^{n})).
		\]
	\end{enumerate}
\end{thm}
\begin{proof}
All the Grassmannians with $T^n$-action are modelled on the same Johnson graph $J(n,k)$ with slightly different congruence relations.
\begin{enumerate}
	\item This is the part\,(1) of Theorem\,\ref{thm:GKMrealGrass}.
	\item From Theorem\,\ref{thm:GKMrealGrass}, the GKM descriptions of even and odd dimensional real Grassmannians have the same congruence relations on the $f_S$ polynomials:
	\[
	f_{S} \equiv f_{S'}\mod \alpha^2_j-\alpha^2_i \quad \text{for $S,S' \in \mathcal{S}$ with $S\cup\{j\}=S'\cup\{i\}$}.
	\] 
	But the odd dimensional real Grassmannian has extra part of $g_S \theta$ with congruence relations:
	\begin{enumerate}
		\item $g_{S} \equiv 0 \mod \prod_{i=1}^{n}\alpha_i$\quad for every $S$
		\item $g_{S} \equiv g_{S'}\mod \alpha^2_j-\alpha^2_i$\quad for $S,S' \in \mathcal{S}$ with $S\cup\{j\}=S'\cup\{i\}$.
	\end{enumerate}	
	The first set of congruence relations means that 
	\[g_{S}=\big(\prod_{i=1}^{n}\alpha_i\big) \cdot h_{S}\] 
	for a polynomial $h_{S}\in \Q[\alpha_1,\ldots,\alpha_n]$ and for every $S$. Substitute into the second set of congruence relations, and note that $\prod_{i=1}^{n}\alpha_i$ is coprime with $\alpha^2_j-\alpha^2_i$, then we get
	\[
	h_{S} \equiv h_{S'}\mod \alpha^2_j-\alpha^2_i \quad \text{ for $S,S' \in \mathcal{S}$ with $S\cup\{j\}=S'\cup\{i\}$}
	\] 
	exactly the same as the congruence relations on the $f_S$ polynomials. Denote
	\[
	r^T=\big((\prod_{i=1}^{n}\alpha_i) \theta\big)_{S\in \mathcal{S}}
	\]
	which has $(r^T)^2=0$ because $\theta$ is the unit volume form of $S^1$, and has degree $2n+1$ because each $\alpha_i$ is of degree $2$ in cohomology. Then we can write 
	\[
	(f_S+g_S\theta)_{S\in \mathcal{S}} = (f_S)_{S\in \mathcal{S}} + r^T\cdot (h_S)_{S\in \mathcal{S}}.
	\]
	This establishes the bijection
	\[
	H^*_{T^n}(G_{2k+1}(\R^{2n+2})) \cong H^*_{T^n}(G_{2k}(\R^{2n}))[r^T]/(r^T)^2
	\]
	which can be easily verified to be a $\Q[\alpha_1,\ldots,\alpha_n]$-algebra isomorphism.
	\item From Theorem \ref{thm:GKMrealGrass}, the GKM description of even dimensional real Grassmannians has the congruence relations on the $f_S$ polynomials:
	\[
	f_{S} \equiv f_{S'}\mod \alpha^2_j-\alpha^2_i \quad \text{for $S,S' \in \mathcal{S}$ with $S\cup\{j\}=S'\cup\{i\}$}
	\] 
	which automatically satisfy the congruence relations on the $f_S$ polynomials for the complex Grassmannians in Theorem \ref{thm:GKMcplxGrass}:
	\[
	f_{S} \equiv f_{S'}\mod \alpha_j-\alpha_i \quad \text{for $S,S' \in \mathcal{S}$ with $S\cup\{j\}=S'\cup\{i\}$}.
	\]
	This establishes the injection
	\[
	H^*_{T^n}(G_{2k}(\R^{2n})) \hookrightarrow H^*_{T^n}(G_{k}(\C^{n}))
	\]
	which is also easy to verify as a $\Q[\alpha_1,\ldots,\alpha_n]$-algebra monomorphism.
\end{enumerate}
\end{proof}

\begin{rmk}
	Those $\Q[\alpha_1,\ldots,\alpha_n]$-algebra isomorphisms in Theorem\,\ref{thm:AllGrass} give ring isomorphisms among ordinary cohomology of real Grassmannians. But the ordinary version $H^*(G_{2k}(\R^{2n})) \rightarrow H^*(G_{k}(\C^{n}))$ is not injective simply due to fact that $G_{2k}(\R^{2n})$ is of dimension $4k(n-k)$, twice the real dimension of $G_{k}(\C^{n})$.
\end{rmk}

\begin{thm}[Canonical basis of even dimensional real Grassmannians]\label{thm:RealSchub}
	There is a self-indexing Morse function on $\mathcal{S}$
	\[
	\psi: \mathcal{S} \longrightarrow \R : S \longmapsto 4(\sum_{i\in S} i) -2k(k+1)
	\]
	and a canonical class $\sigma_S \in H^{\psi(S)}_{T^n}(G_{2k}(\R^{2n}),\Q)$ for each $S\in \mathcal{S}$ such that
	\begin{enumerate}
		\item $\sigma_S$ is supported upward, i.e. $\sigma_S(S')=0$ if $\psi(S')\leq \psi(S)$
		\item $\sigma_S(S)=\prod' (\alpha^2_j - \alpha^2_i)$ where the product is taken over the weights at $S$ connecting to $S'$ with $\psi(S')<\psi(S)$
	\end{enumerate}
	Moreover, $\{\sigma_S, S\in \mathcal{S}\}$ give an additive $\Q[\alpha_1,\ldots,\alpha_n]$-basis of $H^*_{T^n}(G_{2k}(\R^{2n}),\Q)$.
\end{thm}
\begin{proof}
	By Theorem \ref{thm:AllGrass}, we can identify $H^*_{T^n}(G_{2k}(\R^{2n}))$ with its embedded image in $H^*_{T^n}(G_{k}(\C^{n}))$. Recall from Theorem \ref{thm:CplxSchub} on the canonical classes of complex Grassmannian $G_{k}(\C^{n})$, we used the function $\phi=\frac{\psi}{2}$, hence both $\psi$ and $\phi$ define the same partial order on $\mathcal{S}$. Moreover, there is a basis $\tau_S$ of $H^*_{T^n}(G_{k}(\C^{n}))$, such that
	\begin{enumerate}
		\item $\tau_S$ is supported upward, i.e. $\tau_S(S')=0$ if $\phi(S')\leq \phi(S)$
		\item $\tau_S(S)=\prod' (\alpha_j - \alpha_i)$ where the product is taken over the weights at $S$ connecting to $S'$ with $\phi(S')<\phi(S)$
	\end{enumerate}	
	Let's introduce the ring homomorphism:
	\[
	Sq: \Q[\alpha_1,\ldots,\alpha_n] \rightarrow \Q[\alpha_1,\ldots,\alpha_n] : f(\alpha_1,\ldots,\alpha_n) \mapsto f(\alpha^2_1,\ldots,\alpha^2_n).
	\]
	If $f_S \equiv f_{S'} \mod \alpha_j-\alpha_i$, i.e. $f_S - f_{S'}$ is a multiple of $\alpha_j-\alpha_i$, then $Sq(f_S) - Sq(f_{S'})=Sq(f_S - f_{S'})$ is a multiple of $Sq(\alpha_j-\alpha_i)=\alpha^2_j-\alpha^2_i$, i.e. $Sq(f_S) \equiv Sq(f_{S'}) \mod \alpha^2_j-\alpha^2_i$. The homomorphism $Sq$ not only refines the congruence relations of $H^*_{T^n}(G_{k}(\C^{n}))$, but also has image in $H^*_{T^n}(G_{2k}(\R^{2n}))$, i.e. $Sq(H^*_{T^n}(G_{k}(\C^{n}))) \subseteq H^*_{T^n}(G_{2k}(\R^{2n}))$.
	Now we can define $\sigma_S = Sq(\tau_S) \in H^*_{T^n}(G_{2k}(\R^{2n}))$, and we see that this collection of classes satisfies the required properties of being supported upward and $\sigma_S(S)=\prod' (\alpha^2_j - \alpha^2_i)$ over weights at $S$ connecting to $S'$ with $\psi(S')<\psi(S)$.
	
	According to Guillemin\&Zara (\cite{GZ03} pp.\,125, Remark of Thm\,2.1), $\{\sigma_S\}$ give an additive basis of $H^*_{T^n}(G_{2k}(\R^{2n}))$.
\end{proof}

Since we have proved $H^*_{T^n}(G_{2k+1}(\R^{2n+2})) \cong H^*_{T^n}(G_{2k}(\R^{2n}))[r^T]/(r^T)^2$ in Theorem \ref{thm:AllGrass}, then
\begin{cor}[Canonical basis of odd dimensional real Grassmannians]
	 $\sigma_S$ and $r^T\sigma_S$ give an additive $\Q[\alpha_1,\ldots,\alpha_n]$-basis of $H^*_{T^n}(G_{2k+1}(\R^{2n+2}))$.
\end{cor}

\begin{rmk}
	In the case of complex Grassmannian $G_{k}(\C^{n})$, a subset $S\subseteq \{1,2,\ldots,n\}$ with elements $ i_1<i_2<\cdots<i_k$ corresponds to Schubert symbol $(i_1-1,i_2-2,\ldots,i_k-k)$; there could be correspondences for real Grassmannians
	\begin{enumerate}
		\item For even dimensional Grassmannians $G_{2k}(\R^{2n}),G_{2k}(\R^{2n+1})$, let $S$ consist of $ i_1<i_2<\cdots<i_k$, then the $T^n$-fixed point $\oplus_{i\in S} \R^2_{[\alpha_i]}$, with pivot positions $(2i_1-1,2i_1, 2i_2-1,2i_2,\ldots,2i_k-1,2i_k)$ in its reduced echelon form, will correspond to Schubert symbol $(2i_1-2,2i_1-2, 2i_2-4,2i_2-4,\ldots,2i_k-2k,2i_k-2k)$. 
		\item For even dimensional Grassmannian $G_{2k+1}(\R^{2n+1})$, the $T^n$-fixed point $\R_0 \oplus (\oplus_{i\in S} \R^2_{[\alpha_i]})$, with pivot positions $(1,2i_1,2i_1+1, 2i_2,2i_2+1,\ldots,2i_k,2i_k+1)$ in its reduced echelon form, will also correspond to Schubert symbol $(2i_1-2,2i_1-2, 2i_2-4,2i_2-4,\ldots,2i_k-2k,2i_k-2k)$.
		\item For odd dimensional Grassmannian $G_{2k+1}(\R^{2n+2})$, besides the above Schubert symbols $(2i_1-2,2i_1-2, 2i_2-4,2i_2-4,\ldots,2i_k-2k,2i_k-2k)$, there is the class $r^T$, which is conjectured by Casian\&Kodama \cite{CK} to be the Schubert class with the hook Young diagram $1^{2k}\times (2(n-k)+1)$ of symbol $(1,\ldots,1,2(n-k)+1)$ where there are $2k$ copies of $1$. Following this conjecture, we can guess that a class $r^T\sigma_S$ with $S$ given by $i_1<i_2<\cdots<i_k$, corresponds to the Schubert symbol $(2i_1-1,2i_1-1, 2i_2-3,2i_2-3,\ldots,2i_k-2k+1,2i_k-2k+1,2(n-k)+1)$.
	\end{enumerate} 
\end{rmk}

Recall from Subsection \ref{subsec:BGKM} of the equivariant Littlewood-Richardson rule for complex Grassmannian $G_k(\C^n)$
\[
\tau_{S}\tau_{S'} =\sum_{S''}N_{S,S'}^{S''}\tau_{S''} 
\]
where $N_{S,S'}^{S''}\in \Q[\alpha_1,\ldots,\alpha_n]$. If we apply the ring homomorphism $Sq$ on both sides, then we get:
\begin{thm}[Equivariant Littlewood-Richardson coefficients for real Grassmannians]
	The equivariant Littlewood-Richardson coefficients for real Grassmannian $G_{2k}(\R^{2n})$ satisfy
	\[
	\sigma_{S}\sigma_{S'} =\sum_{S''}Sq(N_{S,S'}^{S''})\sigma_{S''}
	\]
	where $Sq(N_{S,S'}^{S''}) \in \Q[\alpha^2_1,\ldots,\alpha^2_n]$ is obtained from $N_{S,S'}^{S''}$ by replacing $\alpha_i$ to be $\alpha_i^2$. 
\end{thm}

\begin{rmk}
Since $Sq$ keeps constant term unchanged, the Littlewood-Richardson rules for ordinary cohomology of complex Grassmannian $G_k(\C^n)$ and real Grassmannian $G_{2k}(\R^{2n})$ are the same.
\end{rmk}

\subsection{Leray-Borel description of real Grassmannians}
Similar to Leray-Borel description of equivariant (ordinary) cohomology of complex Grassmannians using equivariant (ordinary) Chern classes, we will show there is a Leray-Borel description of equivariant (ordinary) cohomology of real Grassmannians using equivariant (ordinary) Pontryagin classes.

\subsubsection{Equivariant Pontryagin classes}
The $T^n$ actions on $\R^{2n}=\oplus_{i=1}^{n} \R^2_{[\alpha_i]}$, $\R^{2n+1}=(\oplus_{i=1}^{n} \R^2_{[\alpha_i]})\oplus \R_0$ and $\R^{2n+2}=(\oplus_{i=1}^{n} \R^2_{[\alpha_i]})\oplus \R^2_0$ induce actions on $G_{2k}(\R^{2n})$, $G_{2k}(\R^{2n+1})$, $G_{2k+1}(\R^{2n+1})$ and $G_{2k+1}(\R^{2n+2})$. These actions induce further actions on the canonical bundles $\gamma$ and complementary bundles $\bar{\gamma}$ over those Grassmannians. Then we can consider their equivariant Pontryagin classes $p^T=p^T(\gamma)$ and $\bar{p}^T=p^T(\bar{\gamma})$.

First, let's compute a warm-up example of equivariant Pontryagin classes.

\begin{lem}\label{lem:Pont}
	The total equivariant Pontryagin class of the vector space $\R^2_{[\alpha]}$ with weight $[\alpha] \in \mathfrak{t}^*_\Z/\pm 1$ over a point is $1+\alpha^2$.
\end{lem}
\begin{proof}
	Think of the elements of $\R^2_{[\alpha]}$ as $2\times 1$ column vectors. For a Lie algebra element $\xi \in \mathfrak{t}$, the action of its group element $\exp(\xi) \in T$ on $\R^2_{[\alpha]}$ is given as a $2\times 2$ real matrix 
	\[
	\begin{pmatrix}
	\cos(\alpha(\xi)) & -\sin(\alpha(\xi))\\
	\sin(\alpha(\xi)) & \cos(\alpha(\xi))
	\end{pmatrix}
	\qquad 
	\textup{or}
	\qquad
	\begin{pmatrix}
	\cos(\alpha(\xi)) & \sin(\alpha(\xi))\\
	-\sin(\alpha(\xi)) & \cos(\alpha(\xi))
	\end{pmatrix}.	
	\]
	Tensoring $\R^2_{[\alpha]}$ over $\R$-coefficients with $\C$ means that we can treat the above real matrices as complex matrices. Since both of them have the same characteristic function $\lambda^2-2\cos(\alpha(\xi))\lambda+1=(\lambda-e^{\sqrt{-1}\alpha(\xi)})(\lambda-e^{-\sqrt{-1}\alpha(\xi)})$, the two real matrices have the same diagonalization over $\C$-coefficients:
	\[
	\begin{pmatrix}
	e^{\sqrt{-1}\alpha(\xi)} & 0\\
	0 & e^{-\sqrt{-1}\alpha(\xi)}
	\end{pmatrix}
	\]
	i.e. the $T$-action on the complex vector space $\R^2_{[\alpha]}\otimes_\R \C$ has weights $\alpha$ and $-\alpha$. Therefore, $c^T(\R^2_{[\alpha]}\otimes_\R \C)=(1-\alpha)(1+\alpha)=1-\alpha^2$. Following Milnor-Stasheff's convention of signs, we get $p^T(\R^2_{[\alpha]})=1+\alpha^2$.
\end{proof}

Second, let's specify the equivariant Pontryagin classes of canonical bundles, complementary bundles and tangent bundles of real Grassmannians in GKM description at each fixed point or circle of the real Grassmannians. 

\begin{prop}\label{prop:Pont}
	For all the four real Grassmannians $G_{2k}(\R^{2n})$, $G_{2k}(\R^{2n+1})$, $G_{2k+1}(\R^{2n+1})$ and $G_{2k+1}(\R^{2n+2})$ with $T^n$-actions, the equivariant Pontryagin classes $p^T=p^T(\gamma)$ and $\bar{p}^T=p^T(\bar{\gamma})$ of the canonical bundle and complementary bundle localized at each fixed point or circle indexed as a $k$-element subset $S \in \{1,\ldots,n\}$ are
	\begin{align*}
	p^T|_S &= p^T(\gamma|_S)=\prod_{i\in S} (1+\alpha^2_i)\\
	\bar{p}^T|_S &= p^T(\bar{\gamma}|_S)=\prod_{j\not\in S} (1+\alpha^2_j)
	\end{align*}
	with the relation $p^T\bar{p}^T = \prod_{i=1}^{n}(1+\alpha^2_i)$. The equivariant Pontryagin classes of the tangent bundles are given at each fixed point or circle $S$ as
	\begin{align*}
	p^T(TG_{2k}(\R^{2n}))|_S&=\prod_{i \in S}\prod_{j \not \in S} \big[(1+(\alpha_j-\alpha_i)^2)(1+(\alpha_j+\alpha_i)^2)\big]\\
	p^T(TG_{2k}(\R^{2n+1}))|_S&=\prod_{i \in S}\prod_{j \not \in S} \big[(1+(\alpha_j-\alpha_i)^2)(1+(\alpha_j+\alpha_i)^2)\big] \prod_{i \in S} (1+\alpha_i^2)\\
	p^T(TG_{2k+1}(\R^{2n+1}))|_S&=\prod_{i \in S}\prod_{j \not \in S} \big[(1+(\alpha_j-\alpha_i)^2)(1+(\alpha_j+\alpha_i)^2)\big] \prod_{j \not \in S} (1+\alpha_j^2)\\
	p^T(TG_{2k+1}(\R^{2n+2}))|_S&=\prod_{i \in S}\prod_{j \not \in S} \big[(1+(\alpha_j-\alpha_i)^2)(1+(\alpha_j+\alpha_i)^2)\big] \prod_{i \in S} (1+\alpha_i^2) \prod_{j \not \in S} (1+\alpha_j^2).
	\end{align*}
\end{prop}
\begin{proof}
	For $G_{2k}(\R^{2n})$, at each fixed point $S$, we have $\gamma|_S=\oplus_{i\in S} \R^2_{[\alpha_i]}$ and $\bar{\gamma}|_S=\oplus_{j \not\in S} \R^2_{[\alpha_j]}$, and furthermore $\gamma|_S\oplus\bar{\gamma}|_S=\oplus_{i=1}^{n} \R^2_{[\alpha_i]}$. The claimed expressions of the localized Pontryagin classes then follow from the Lemma \ref{lem:Pont}. The cases of $G_{2k}(\R^{2n+1})$, $G_{2k+1}(\R^{2n+1})$ and $G_{2k+1}(\R^{2n+2})$ are similar. For the Pontryagin classes of the tangent bundles localized at each fixed point or circle, we can apply Lemma \ref{lem:Pont} to the weight decompositions (see Subsubsection \ref{subsubsec:IsoWeights}) of tangent bundles at each fixed point or circle. 
\end{proof}

\subsubsection{Characteristic basis of real Grassmannians}
Think of $p^T$ and $\bar{p}^T$ as elements of the embedded image of $H^*_{T^n}(G_{2k}(\R^{2n}))$ in $H^*_{T^n}(G_{k}(\C^{n}))$ using GKM description on the Johnson graph $J(n,k)$. If we compare the localized expressions of $p^T$ and $\bar{p}^T$ with $c^T$ and $\bar{c}^T$ in Subsection \ref{subsec:BGKM}, we get the formula
\[
p^T=Sq(c^T) \qquad \bar{p}^T=Sq(\bar{c}^T)
\]
where the homomorphism $Sq$ is defined in Theorem \ref{thm:RealSchub}.

Recall in Subsection \ref{subsec:BGKM}, we discussed the transformations $K,\bar{K}$ between the characteristic monomials $(c^T)^I=(c_1^T)^{i_1}\cdots(c_k^T)^{i_k}$ in Leray-Borel description and the canonical classes $\tau_S$ in GKM description:
\begin{align*}
(c^T)^I &= \sum_S K^I_S \tau_S\\
\tau_S &= \sum_I \bar{K}_I^S (c^T)^I
\end{align*}
where $K^I_S,\bar{K}_I^S \in \Q[\alpha_1,\ldots,\alpha_n]$. Apply the homomorphism $Sq$ and recall $\sigma_S=Sq(\tau_S)$ from Theorem \ref{thm:RealSchub}, then
\begin{align*}
(p^T)^I &= \sum_S Sq(K^I_S) \sigma_S\\
\sigma_S &= \sum_I Sq(\bar{K}_I^S) (p^T)^I
\end{align*}
where $Sq(K^I_S),Sq(\bar{K}_I^S) \in \Q[\alpha_1,\ldots,\alpha_n]$.

Since $\{\sigma_S\}$ give a basis of $H^*_{T^n}(G_{2k}(\R^{2n}))$, the above transformations imply:

\begin{thm}[Equivariant characteristic basis of real Grassmannians]\label{thm:EquivCharReal}
	The set of monomials $(p^T_1)^{r_1}(p^T_2)^{r_2}\cdots (p^T_k)^{r_k}$ satisfying the condition $\sum_{i=1}^{k} r_i \leq n-k$ forms an additive $H^*_T(pt)$-basis for $H^*_{T^n}(G_{2k}(\R^{2n}))\cong H^*_{T^n}(G_{2k}(\R^{2n+1}))\cong H^*_{T^n}(G_{2k+1}(\R^{2n+1}))$. Together with the set of monomials $r^T\cdot(p^T_1)^{r_1}(p^T_2)^{r_2}\cdots (p^T_k)^{r_k}$, they form an additive $H^*_T(pt)$-basis for $H^*_{T^n}(G_{2k+1}(\R^{2n+2}))$.
\end{thm}

Now we can give the Leray-Borel description for equivariant cohomology of real Grassmannians.
\begin{thm}[Equivariant Leray-Borel description of even dimensional real Grassmannians]\label{thm:EquivBReal}
	For the even dimensional real Grassmannians $G_{2k}(\R^{2n})$, $G_{2k}(\R^{2n+1})$ and $G_{2k+1}(\R^{2n+1})$ with $T^n$-actions, their equivariant cohomology is the same:
	\[
	H^*_T(G_{2k}(\R^{2n}),\Q)\cong\frac{\Q[\alpha_1,\alpha_2,\ldots,\alpha_n][p^T_1,p^T_2,\ldots,p^T_k;\bar{p}^T_1,\bar{p}^T_2,\ldots,\bar{p}^T_{n-k}]}{p^T\bar{p}^T = \prod_{i=1}^{n}(1+\alpha^2_i)}.
	\]
\end{thm}
\begin{proof}
	Apply the $Sq$ to the generators and relations in the Leray-Borel description of $H^*_T(G_k(\C^n))$.
\end{proof}

\begin{thm}[Equivariant Leray-Borel description of odd dimensional real Grassmannians]\label{thm:EquivBReal2}
	For the odd dimensional real Grassmannian $G_{2k+1}(\R^{2n+2})$ with $T^n$-actions, the equivariant cohomology is:
	\[
	H^*_T(G_{2k+1}(\R^{2n+2}),\Q)\cong\frac{\Q[\alpha_1,\alpha_2,\ldots,\alpha_n][p^T_1,p^T_2,\ldots,p^T_k;\bar{p}^T_1,\bar{p}^T_2,\ldots,\bar{p}^T_{n-k};r^T]}{p^T\bar{p}^T = \prod_{i=1}^{n}(1+\alpha^2_i),\,(r^T)^2=0}.
	\]
\end{thm}

For a $T^n$-equivariantly formal space $M$, we can recover the ordinary cohomology from the equivariant cohomology by $H^*(M,\Q)=H_T^*(M,\Q)\otimes_{\Q[\alpha_1,\ldots,\alpha_n]} \Q$ where $\Q$ has a $\Q[\alpha_1,\ldots,\alpha_n]$-algebra structure from the constant-term morphism $\Q[\alpha_1,\ldots,\alpha_n]\rightarrow \Q: f(\alpha_1,\ldots,\alpha_n)\mapsto f(0)$. Therefore, the above Theorem \ref{thm:EquivCharReal}, \ref{thm:EquivBReal}, \ref{thm:EquivBReal2} have ordinary versions by ignoring the $\alpha_i$. 

Let $r\in H^{2n+1}(G_{2k+1}(\R^{2n+2}),\Q)$ be the ordinary image of the $r^T\in H^{2n+1}_T(G_{2k+1}(\R^{2n+2}),\Q)$.

\begin{cor}[Ordinary characteristic basis of real Grassmannians]
	The set of monomials $(p_1)^{r_1}(p_2)^{r_2}\cdots (p_k)^{r_k}$ satisfying the condition $\sum_{i=1}^{k} r_i \leq n-k$ forms an additive basis for $H^*(G_{2k}(\R^{2n}))\cong H^*(G_{2k}(\R^{2n+1}))\cong H^*(G_{2k+1}(\R^{2n+1}))$. Together with the set of monomials $r\cdot(p_1)^{r_1}(p_2)^{r_2}\cdots (p_k)^{r_k}$, they form an additive basis for $H^*(G_{2k+1}(\R^{2n+2}))$.
\end{cor}

\begin{cor}[Ordinary Leray-Borel description of real Grassmannians]
	For the even dimensional real Grassmannians $G_{2k}(\R^{2n})$, $G_{2k}(\R^{2n+1})$ and $G_{2k+1}(\R^{2n+1})$, their cohomology is the same:
	\[
	\frac{\Q[p_1,p_2,\ldots,p_k;\bar{p}_1,\bar{p}_2,\ldots,\bar{p}_{n-k}]}{p\bar{p} = 1}.
	\]
	For the odd dimensional real Grassmannian $G_{2k+1}(\R^{2n+2})$, the cohomology is:
	\[
	\frac{\Q[p_1,p_2,\ldots,p_k;\bar{p}_1,\bar{p}_2,\ldots,\bar{p}_{n-k};r]}{p\bar{p} = 1,\, r^2=0}
	\]
	where $r\in H^{2n+1}(G_{2k+2}(\R^{2n+2}),\Q)$ is the ordinary image of the $r^T\in H^{2n+1}_T(G_{2k+2}(\R^{2n+2}),\Q)$.
\end{cor}

\begin{rmk}
	This explicit Leray-Borel description for the real Grassmannians is stated in Casian\&Kodama \cite{CK}. The even dimensional case is a special case of the Leray-Borel description, and the odd dimensional case is due to Takeuchi \cite{Ta62}. 
\end{rmk}

\begin{rmk}
	For $n\leq 7$, the ordinary cohomology groups of $G_k(\R^n)$ in $\Z$ coefficients were computed by Jungkind \cite{Ju79}. 
\end{rmk}

\vskip 20pt
\section{Equivariant cohomology rings of oriented Grassmannians}
\vskip 15pt
In this section, we give the GKM description and Leray-Borel description of equivariant cohomology rings of oriented Grassmannians, together with the characteristic basis of the additive structure. We use the notation $\tilde{G}_k(\R^n)$ for the Grassmannian of $k$-dimensional oriented subspaces in $\R^n$.

The Pl\"{u}cker embedding of an oriented Grassmannian can be given as follows: for $V\in \tilde{G}_k(\R^n)$, we can choose an ordered orthonormal basis $v_1,\ldots,v_k$ of $V$, then the well-defined wedge product $v_1 \wedge \cdots \wedge v_k \in \tilde{G}_1(\wedge^k \R^n)=S(\wedge^k \R^n)$ in the unit sphere of $\wedge^k \R^n$ gives the embedding $\tilde{G}_k(\R^n) \hookrightarrow S(\wedge^k \R^n)$.

Similar to the case of real Grassmannians, we can consider the $T^n$-action on $\R^{2n}=\oplus_{i=1}^{n} \R^2_{[\alpha_i]}$, $\R^{2n+1}=(\oplus_{i=1}^{n} \R^2_{[\alpha_i]})\oplus \R_0$ and $\R^{2n+2}=(\oplus_{i=1}^{n} \R^2_{[\alpha_i]})\oplus \R^2_0$ for their decompositions into weighted subspaces, where $\alpha_1,\ldots,\alpha_n$ are the standard basis of $\mathfrak{t}_\Z^*$. These actions induce $T^n$ actions on $\tilde{G}_{2k}(\R^{2n})$, $\tilde{G}_{2k}(\R^{2n+1})$, $\tilde{G}_{2k+1}(\R^{2n+1})$ and $\tilde{G}_{2k+1}(\R^{2n+2})$. More specifically, each $t \in T$ maps $v_1 \wedge \cdots \wedge v_l$ where $l=2k,2k+1$, to $t\cdot v_1 \wedge \cdots \wedge t \cdot v_l$, and it is easy to check the map is independent from the choice of a positive orthonormal basis $v_1,\ldots,v_k$. 

Also since there are natural $T^n$-diffeomorphisms $\tilde{G}_{2k}(\R^{2n+1})\cong \tilde{G}_{2n-2k+1}(\R^{2n+1})$ identifying the second and the third types of real Grassmannians, in some discussions we will only consider the three cases of $\tilde{G}_{2k}(\R^{2n})$, $\tilde{G}_{2k}(\R^{2n+1})$ and $\tilde{G}_{2k+1}(\R^{2n+2})$.

\subsection{Oriented Grassmannians as $2$-covers over real Grassmannians}
There are natural $2$-coverings of oriented Grassmannians over real Grassmannians $\pi:\tilde{G}_k(\R^n)\rightarrow G_k(\R^n): v_1 \wedge \cdots \wedge v_k \mapsto \mathrm{Span}_\R(v_1,\ldots,v_k)$ which induces a pull-back morphism $\pi^*: H^*(G_k(\R^n))\rightarrow H^*(\tilde{G}_k(\R^n))$ between their cohomology. The non-trivial deck transformation is defined by reversing orientations $\rho: \tilde{G}_k(\R^n)\rightarrow\tilde{G}_k(\R^n):v_1 \wedge \cdots \wedge v_k \mapsto -(v_1 \wedge \cdots \wedge v_k)$ which induces an isomorphism $\rho^*: H^*(\tilde{G}_k(\R^n)) \rightarrow H^*(\tilde{G}_k(\R^n))$. Both $\pi$ and $\rho$ commute with the $T$-actions that we introduced on the oriented Grassmannians and real Grassmannians. 

For covering maps between compact spaces, or equivalently for free actions of finite groups, there is a well-known fact relating their cohomology in rational coefficients:
\begin{lem}
	Let $\pi:X\rightarrow Y$ be a covering between compact topological spaces with a finite deck transformation group $G$ which also acts on the cohomology $H^*(X,\Q)$. Then $\pi^*: H^*(Y,\Q) \rightarrow H^*(X,\Q)$ is injective with image $H^*(X,\Q)^G$. This conclusion is also true for equivariant cohomology if a torus $T$ acts on $X$ and commutes with the action of $G$. 
\end{lem}
\begin{proof}
	For a cocycle $c$ of $X$, the averaged cocycle $\frac{1}{|G|}\sum_{g\in G} gc$ is invariant under $G$-action, hence comes from a cocycle of $Y$. Consider the averaging map $\pi_*: H^*(X,\Q) \rightarrow H^*(Y,\Q):[c]\mapsto \frac{1}{|G|}[\sum_{g\in G} gc]$, then the composition $\pi_* \pi^*$ is the identity map on $H^*(Y,\Q)$, hence $\pi^*$ is injective. Note that every cohomology class in $H^*(X,\Q)^G$ can be represented by a $G$-invariant cocycle using the averaging method. This proves the image of $\pi^*$ is exactly $H^*(X,\Q)^G$.
	
	For the $T^n$-equivariant version, though the Borel construction $X\times_{T^n} (S^\infty)^n,Y\times_{T^n} (S^\infty)^n$ is not compact, we can apply the ordinary version of current Lemma to the compact approximations $X\times_{T^n} (S^N)^n, Y\times_{T^n} (S^N)^n$ for $N\rightarrow \infty$.
\end{proof}

\begin{rmk}
	For the averaging method to work, we can relax the $\Q$ coefficients to be any coefficient ring that contains $\frac{1}{|G|}$. In $\R$ coefficients, the ordinary and equivariant de Rham theory together with the averaging method give a proof without using compact approximations. 
\end{rmk}

Applying this Lemma to the oriented Grassmannians as $T$-equivariant $2$-covers over real Grassmannians, we get
\begin{prop}\label{prop:OrientVSReal}
	The pull-back morphisms of ordinary and equivariant cohomology
	\[
	\pi^*: H^*(G_k(\R^n)) \hookrightarrow H^*(\tilde{G}_k(\R^n)) \qquad \text{and} \qquad H^*_T(G_k(\R^n)) \hookrightarrow H^*_T(\tilde{G}_k(\R^n))
	\]
	are both injective. Moreover,
	\[
	\pi^*(H^*(G_k(\R^n))) = H^*(\tilde{G}_k(\R^n))^{\Z/2} \qquad \text{and} \qquad \pi^*(H^*_T(G_k(\R^n))) = H^*_T(\tilde{G}_k(\R^n))^{\Z/2}
	\]
	identifies cohomology of real Grassmannians as the ${\Z/2}$-invariant subrings of cohomology of oriented Grassmannians, or equivalently as the $+1$-eigenspaces of $\rho^*$ on cohomology of oriented Grassmannians.
\end{prop}

For odd dimensional oriented Grassmannian $\tilde{G}_{2k+1}(\R^{2n+2})$, the deck transformation $\rho:\tilde{G}_{2k+1}(\R^{2n+2}) \rightarrow \tilde{G}_{2k+1}(\R^{2n+2}): v_1 \wedge \cdots \wedge v_{2k+1} \mapsto -(v_1 \wedge \cdots \wedge v_{2k+1})=(-v_1) \wedge \cdots \wedge (-v_{2k+1})$ is induced from the antipodal map $A:\R^{2n+2}\rightarrow \R^{2n+2}: v\mapsto -v$ which is homotopic to the identity map on $\R^{2n+2}$ via
\[
\begin{pmatrix}
\begin{matrix}
\cos \theta & -\sin \theta \\
\sin \theta & \cos \theta
\end{matrix}
& &\\
& 
\ddots
&\\
& &
\begin{matrix}
\cos \theta & -\sin \theta \\
\sin \theta & \cos \theta
\end{matrix}
\end{pmatrix}
\]
which is actually $T^n$-equivariant and further induces $T^n$-equivariant homotopy between $\rho$ and $id$ on $\tilde{G}_{2k+1}(\R^{2n+2})$.

\begin{thm}[Relations between odd dimensional oriented and real Grassmannians]\label{thm:OddGrass}
	Since $\rho^*=id$ on ordinary and equivariant cohomology of odd dimensional oriented Grassmannians, we have
	\begin{align*}
	H^*(\tilde{G}_{2k+1}(\R^{2n+2})) &= H^*(\tilde{G}_{2k+1}(\R^{2n+2}))^{\Z/2} \cong H^*(G_{2k+1}(\R^{2n+2}))\\  H^*_T(\tilde{G}_{2k+1}(\R^{2n+2})) &= H^*_T(\tilde{G}_{2k+1}(\R^{2n+2}))^{\Z/2} \cong H^*_T(G_{2k+1}(\R^{2n+2})).
	\end{align*}
\end{thm}

\begin{cor}
	The Poincar\'e series of $\tilde{G}_{2k+1}(\R^{2n+2})$ are
	\[
	P_{\tilde{G}_{2k+1}(\R^{2n+2})}(t)=P_{G_{2k+1}(\R^{2n+2})}(t)=(1+t^{2n+1})P_{G_{k}(\C^{n})}(t^2).
	\]
\end{cor}

The relations between cohomology of oriented and real Grassmannians in even dimensions are more delicate. In next two subsections, we will try to understand the $\rho^*$-action on finer structures of the cohomology rings of oriented Grassmannians.

\subsection{GKM description of oriented Grassmannians}
The GKM description of oriented Grassmannians is very similar to that of real Grassmannians in previous section. Hence most of the details will be omitted but referred to those of real Grassmannians. 

\subsubsection{Orientations and Euler classes of canonical bundle and complementary bundle}\label{subsub:Euler}
The preferred orientation on every oriented $k$-dimensional subspace in $\R^n$ brings new invariants. 

For example, the $T^n$-fixed points of $\tilde{G}_{2k}(\R^{2n})$ as $2k$-dimensional $T^n$-subrepresentation of $\R^{2n}=\oplus_{i=1}^{n} \R^2_{[\alpha_i]}$ are of the form $\oplus_{i\in S} \R^2_{[\alpha_i]}$ where $S$ is a $k$-element subset of $\{1,2,\ldots,n\}$. Though an orientation on $\oplus_{i\in S} \R^2_{[\alpha_i]}$ can not specify the signs of the individual weights $\alpha_i,i\in S$, it does specify the sign of the product of weights as either $\prod_{i\in S}\alpha_i$ or $-\prod_{i\in S}\alpha_i$, which is exactly the equivariant Euler class of an oriented $T$-representation over a point. We will denote $V_{S_+},V_{S_-}$ for the $T$-subrepresentation $\oplus_{i\in S} \R^2_{[\alpha_i]}$ with equivariant Euler classes $e^T$ as $\prod_{i\in S}\alpha_i,-\prod_{i\in S}\alpha_i$ respectively. Similarly, for $\tilde{G}_{2k}((\oplus_{i=1}^{n} \R^2_{[\alpha_i]})\oplus \R_0)$ we can introduce the same notations with the fixed points $V_{S_\pm}=(\oplus_{i\in S}\R^2_{[\alpha_i]},\pm \prod_{i\in S}\alpha_i)$. 

For $\tilde{G}_{2k+1}((\oplus_{i=1}^{n} \R^2_{[\alpha_i]})\oplus \R_0)$, the fixed points are of the form $(\oplus_{i\in S} \R^2_{[\alpha_i]})\oplus \R_0$ which has equivariant Euler class $0$ because of the $0$-weight space $\R_0$. But an orientation on $(\oplus_{i\in S} \R^2_{[\alpha_i]})\oplus \R_0$ gives the complementary $T^n$-subrepresentation $\oplus_{j \not\in S} \R^2_{[\alpha_j]}$ an orientation hence an equivariant Euler class $\bar{e}^T$ either $\prod_{j \not\in S}\alpha_j$ or $-\prod_{j \not\in S}\alpha_j$. We will denote these fixed points as $V_{S_\pm}=((\oplus_{i\in S} \R^2_{[\alpha_i]})\oplus \R_0,\pm \prod_{j \not\in S}\alpha_j)$.

For $\tilde{G}_{2k+1}((\oplus_{i=1}^{n} \R^2_{[\alpha_i]})\oplus \R_0^2)$, a fixed component is of the form $\tilde{C}_S=\{(\oplus_{i\in S} \R^2_{[\alpha_i]})\oplus L_0 \mid L_0\in \tilde{G}_1(\R^2_0)\}\cong S^1$. Since both $(\oplus_{i\in S} \R^2_{[\alpha_i]})\oplus L_0$ and its complement $(\oplus_{j \not\in S} \R^2_{[\alpha_j]})\oplus L_0^\perp$ have a $0$-weight part, the equivariant Euler classes of both $T^n$-subrepresentations are $0$. 

\subsubsection{$1$-skeleta}
We can describe the $1$-skeleta of oriented Grassmannians:
\begin{prop}[$1$-skeleta of oriented Grassmannians]\label{prop:OrientSkeleton}
	The fixed points, isotropy weights and $1$-skeletons of oriented Grassmannians can be given as
	\begin{enumerate}
		\item For $\tilde{G}_{2k}(\R^{2n})$, there are $2\binom{n}{k}$ fixed points of the form $V_{S_\pm}=(\oplus_{i\in S} \R^2_{[\alpha_i]},\pm \prod_{i\in S}\alpha_i)$, where $S$ is a $k$-element subset of $\{1,2,\ldots,n\}$. The isotropy weights at both $V_{S_\pm}$ are $\{[\alpha_j\pm\alpha_i] \mid i\in S,j\not \in S\}$, among which $[\alpha_j-\alpha_i]$ joins $V_{S_\pm}$ via a $2$-sphere to $V_{((S\setminus\{i\})\cup \{j\})_\pm}$, and $[\alpha_j+\alpha_i]$ joins $V_{S_\pm}$ via a $2$-sphere to $V_{((S\setminus\{i\})\cup \{j\})_\mp}$.
		
		\item For $\tilde{G}_{2k}(\R^{2n+1})$, there are $2\binom{n}{k}$ fixed points of the form $V_{S_\pm}=(\oplus_{i\in S} \R^2_{[\alpha_i]},\pm \prod_{i\in S}\alpha_i)$, where $S$ is a $k$-element subset of $\{1,2,\ldots,n\}$. The isotropy weights at both $V_{S_\pm}$ are $\{[\alpha_j\pm\alpha_i] \mid i\in S,j\not \in S\}\cup \{[\alpha_i]\mid i \in S\}$, among which $[\alpha_j-\alpha_i]$ joins $V_{S_\pm}$ via a $2$-sphere to $V_{((S\setminus\{i\})\cup \{j\})_\pm}$, and $[\alpha_j+\alpha_i]$ joins $V_{S_\pm}$ via a $2$-sphere to $V_{((S\setminus\{i\})\cup \{j\})_\mp}$, and $[\alpha_i]$ joins $V_{S_+}$ via a $2$-sphere to $V_{S_-}$.
		
		\item For $\tilde{G}_{2k+1}(\R^{2n+1})$, there are $2\binom{n}{k}$ fixed points of the form $V_{S_\pm}=((\oplus_{i\in S} \R^2_{[\alpha_i]})\oplus \R_0,\pm \prod_{j \not\in S}\alpha_j)$, where $S$ is a $k$-element subset of $\{1,2,\ldots,n\}$. The isotropy weights at both $V_{S_\pm}$ are $\{[\alpha_j\pm\alpha_i] \mid i\in S,j\not \in S\} \cup \{[\alpha_j]\mid j \not \in S\}$, among which $[\alpha_j-\alpha_i]$ joins $V_{S_\pm}$ via a $2$-sphere to $V_{((S\setminus\{i\})\cup \{j\})_\pm}$ and $[\alpha_j+\alpha_i]$ joins $V_{S_\pm}$ via a $2$-sphere to $V_{((S\setminus\{i\})\cup \{j\})_\mp}$, and $[\alpha_j]$ joins $V_{S_+}$ via a $2$-sphere to $V_{S_-}$.
		
		\item For $\tilde{G}_{2k+1}(\R^{2n+2})$, there are $\binom{n}{k}$ fixed circles of the form $\tilde{C}_S=\{(\oplus_{i\in S} \R^2_{[\alpha_i]})\oplus L_0 \mid L_0\in \tilde{G}_1(\R^2_0)\}\cong S^1$, where $S$ is a $k$-element subset of $\{1,2,\ldots,n\}$. The isotropy weights at $\tilde{C}_S$ are $\{[\alpha_j\pm\alpha_i] \mid i\in S,j\not \in S\} \cup \{[\alpha_i]\mid i\in S\}\cup \{[\alpha_j]\mid j \not \in S\}$, among which both $[\alpha_j+\alpha_i]$ and $[\alpha_j-\alpha_i]$ join $\tilde{C}_S$ via a $S^2\times S^1$ to $\tilde{C}_{ (S\setminus\{i\})\cup \{j\}}$, and $[\alpha_i],[\alpha_j]$ join $\tilde{C}_S$ via a $S^3$ to no other fixed circles.		
	\end{enumerate}
\end{prop}
\begin{proof}
	Similar to the case of real Grassmannians.
\end{proof}

\subsubsection{GKM graphs of oriented Grassmannians}
Using the $1$-skeleta of oriented Grassmannians, we can construct their GKM graphs:

\begin{exm}
	We will give some examples of GKM graphs for $\tilde{G}_k(\R^n)$ when $k$ or $n$ is small.
	\begin{enumerate}
		\item $S^{2n}$ as $\tilde{G}_{1}(\R^{2n+1})$ or $\tilde{G}_{2n}(\R^{2n+1})$ 
		\begin{figure}[H]
			\centering
			\begin{subfigure}[b]{.45\textwidth}
				\centering
				\includegraphics{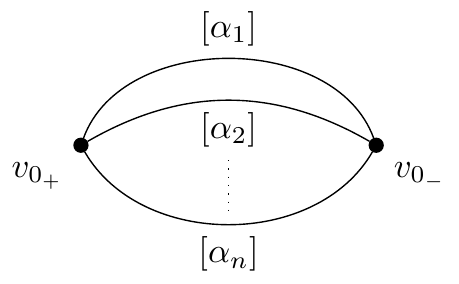}
				\caption{GKM graph for $S^{2n}$}
			\end{subfigure}
			\begin{subfigure}[b]{.45\textwidth}
				\centering
				\includegraphics{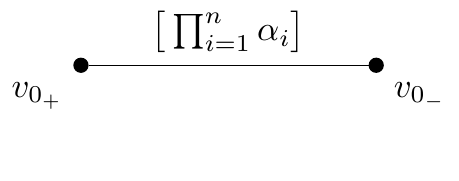}
				\caption{Condensed GKM graph for $S^{2n}$}
			\end{subfigure}
			\caption{GKM graphs for $S^{2n}$}
		\end{figure}
		\item $S^{2n+1}$ as $\tilde{G}_{1}(\R^{2n+2})$ or $\tilde{G}_{2n+1}(\R^{2n+2})$ 
		\begin{figure}[H]
			\centering
			\begin{subfigure}[b]{.45\textwidth}
				\centering
				\includegraphics{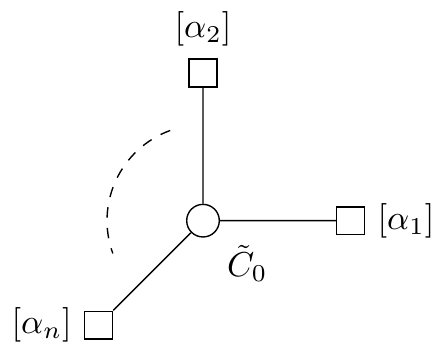}
				\caption{GKM graph for $S^{2n+1}$}
			\end{subfigure}
			\begin{subfigure}[b]{.45\textwidth}
				\centering
				\includegraphics{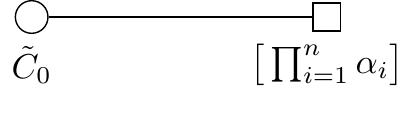}
				\caption{Condensed GKM graph for $S^{2n+1}$}
			\end{subfigure}
			\caption{GKM graphs for $S^{2n+1}$}
		\end{figure}
		\item $\tilde{G}_{2}(\R^{4}),\tilde{G}_{2}(\R^{5}),\tilde{G}_{3}(\R^{5})$ as $\tilde{G}_{2k}(\R^{2n}), \tilde{G}_{2k}(\R^{2n+1}), \tilde{G}_{2k+1}(\R^{2n+1})$ when $k=1,n=2$.
		\begin{figure}[H]
			\centering
			    \begin{subfigure}[b]{0.3\textwidth}
					\centering
					\includegraphics{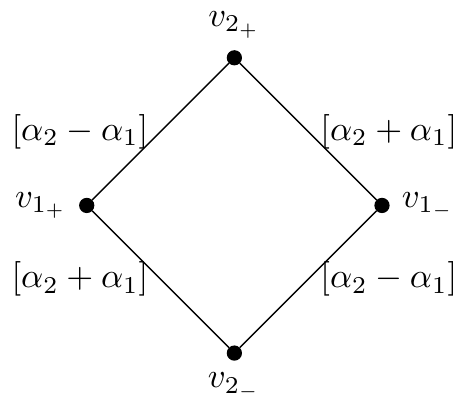}
					\caption{GKM graph for $\tilde{G}_{2}(\R^{4})$}
				\end{subfigure}
				\begin{subfigure}[b]{0.3\textwidth}
					\centering
					\includegraphics{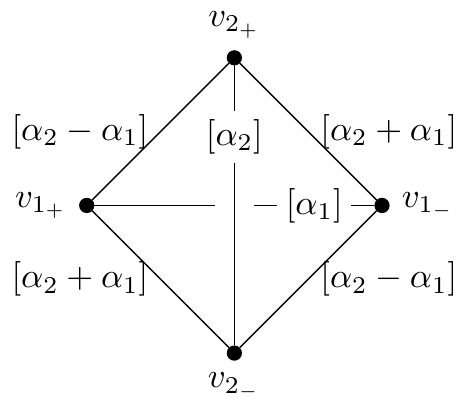}
					\caption{GKM graph for $\tilde{G}_{2}(\R^{5})$}
				\end{subfigure}
				\begin{subfigure}[b]{0.3\textwidth}
					\centering
					\includegraphics{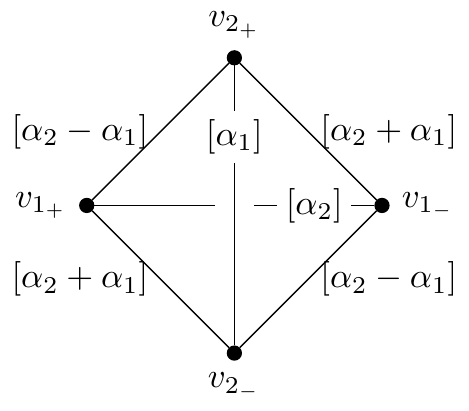}
					\caption{GKM graph for $\tilde{G}_{3}(\R^{5})$}
				\end{subfigure}
			\caption{GKM graphs for $\tilde{G}_{2}(\R^{4}),\tilde{G}_{2}(\R^{5}),\tilde{G}_{3}(\R^{5})$}
		\end{figure}		
		\item $\tilde{G}_{3}(\R^{6})$ as $\tilde{G}_{2k+1}(\R^{2n+2})$ when $k=1,n=2$.
		\begin{figure}[H]
			\centering
			\begin{subfigure}[b]{0.45\textwidth}
				\centering
				\includegraphics{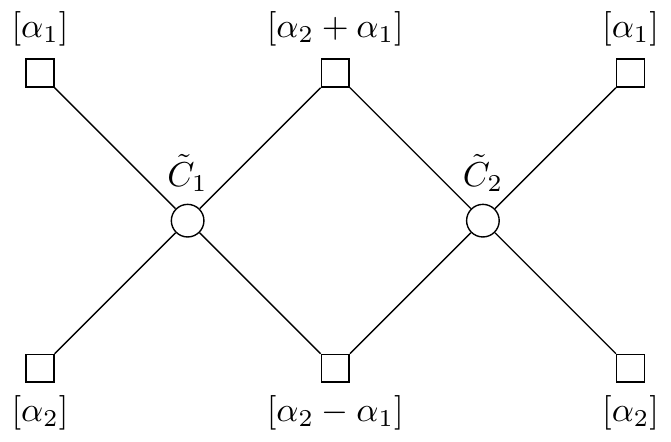}
				\caption{GKM graph for $\tilde{G}_{3}(\R^{6})$}
			\end{subfigure}
			\begin{subfigure}[b]{0.45\textwidth}
				\centering
				\includegraphics{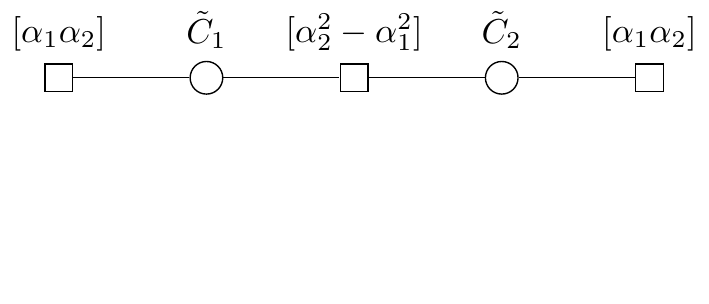}
				\caption{Condensed GKM graph for $\tilde{G}_{3}(\R^{6})$}
			\end{subfigure}
			\caption{GKM graphs for $\tilde{G}_{3}(\R^{6})$}
		\end{figure}
	\end{enumerate}
\end{exm}

\subsubsection{Formality, cohomology and canonical basis of oriented Grassmannians}
\begin{prop}[Equivariant formality of torus actions on oriented Grassmannians]
	The $T^n$-actions on all the four types of oriented Grassmannians $\tilde{G}_{2k}(\R^{2n}), \tilde{G}_{2k}(\R^{2n+1}), \tilde{G}_{2k+1}(\R^{2n+1}), \tilde{G}_{2k+1}(\R^{2n+2})$ are equivariantly formal. All the four oriented Grassmannians have the same total Betti number $2\binom{n}{k}$.
\end{prop}
\begin{proof}
    The even dimensional oriented Grassmannians $\tilde{G}_{2k}(\R^{2n}), \tilde{G}_{2k}(\R^{2n+1}), \tilde{G}_{2k+1}(\R^{2n+1})$ can be viewed as homogeneous spaces of the form $G/H$ with $H$ a compact connected Lie subgroup and of the same rank as the connected compact Lie group $G$. As shown in \cite{GHZ06}, the actions of maximal tori on these homogeneous spaces are equivariantly formal. Their total Betti numbers are the same as their Euler characteristic numbers $|W_G/W_H|$ where $W_G,W_H$ are the Weyl groups of $G$ and $H$. Alternatively, we can compute the total Betti number as the number of fixed points given in Theorem \ref{prop:OrientSkeleton}, namely $2\binom{n}{k}$. The odd dimensional oriented Grassmannian $\tilde{G}_{2k+1}(\R^{2n+2})$ has the same equivariant cohomology as the real Grassmannian $G_{2k+1}(\R^{2n+2})$, hence is also equivariantly formal with total Betti number $2\binom{n}{k}$.
\end{proof}

After verifying GKM conditions and equivariant formality, we can give the GKM description of the torus actions on $\tilde{G}_k(\R^n)$ by applying the even dimensional GKM Theorem as in Guillemin, Holm and Zara \cite{GHZ06} and the odd dimensional GKM-type Theorem \ref{thm:OddGKM}. 

\begin{thm}[GKM description of equivariant cohomology of oriented Grassmannians]\label{thm:GKMorientGrass}
	The following congruence relations are given for any two $k$-element subsets $S,S'\subset \{1,\ldots,n\}$ differed by one element with $S\cup\{j\}=S'\cup\{i\}$.
	\begin{enumerate}
		\item For the oriented Grassmannian $\tilde{G}_{2k}(\R^{2n})$, an element of equivariant cohomology is a set of polynomials $f_{S_\pm} \in \Q[\alpha_1,\ldots,\alpha_n]$ to each vertex $S_\pm$ such that 
		\begin{enumerate}
			\item $f_{S_+} \equiv f_{S'_+}, \quad f_{S_-} \equiv f_{S'_-} \mod \alpha_j-\alpha_i$
			\item $f_{S_+} \equiv f_{S'_-}, \quad f_{S_-} \equiv f_{S'_+} \mod \alpha_j+\alpha_i$.
		\end{enumerate}
		\item For the oriented Grassmannian $\tilde{G}_{2k}(\R^{2n+1})$, an element of equivariant cohomology is a set of polynomials $f_{S_\pm} \in \Q[\alpha_1,\ldots,\alpha_n]$ to each vertex $S_\pm$ such that 
		\begin{enumerate}
			\item $f_{S_+} \equiv f_{S'_+} , \quad f_{S_-} \equiv f_{S'_-}\mod \alpha_j-\alpha_i$
			\item $f_{S_+} \equiv f_{S'_-} , \quad f_{S_-} \equiv f_{S'_+}\mod \alpha_j+\alpha_i$
			\item $f_{S_+} \equiv f_{S_-} \mod \prod_{i'\in S}\alpha_{i'}$.
		\end{enumerate}
		\item For the oriented Grassmannian $\tilde{G}_{2k+1}(\R^{2n+1})$, an element of equivariant cohomology is a set of polynomials $f_{S_\pm} \in \Q[\alpha_1,\ldots,\alpha_n]$ to each vertex $S_\pm$ such that 
		\begin{enumerate}
			\item $f_{S_+} \equiv f_{S'_+} , \quad f_{S_-} \equiv f_{S'_-}\mod \alpha_j-\alpha_i$
			\item $f_{S_+} \equiv f_{S'_-} , \quad f_{S_-} \equiv f_{S'_+}\mod \alpha_j+\alpha_i$
			\item $f_{S_+} \equiv f_{S_-} \mod \prod_{j'\not\in S}\alpha_{j'}$.
		\end{enumerate}
		\item For the odd dimensional oriented Grassmannian $\tilde{G}_{2k+1}(\R^{2n+2})$ with $T^n$-action, an element of equivariant cohomology is a set of polynomial pairs $(f_S, g_S \theta)$ to each $\circ$-vertex $S$ where $\theta$ is the unit volume form of $S^1$ such that
		\begin{enumerate}
			\item $g_{S} \equiv 0 \mod \prod_{i=1}^{n}\alpha_i$
			\item $f_{S} \equiv f_{S'} , \quad g_{S} \equiv g_{S'}\mod \alpha^2_j-\alpha^2_i$.
		\end{enumerate}
	\end{enumerate}
\end{thm}

For odd dimensional oriented Grassmannians, the induced deck transformation $\rho^*: H^*_T(\tilde{G}_{2k+1}(\R^{2n+2}))\rightarrow H^*_T(\tilde{G}_{2k+1}(\R^{2n+2}))$ is the identity map hence acts trivially on the GKM description. Solving the same set of congruence equations as in Theorem \ref{thm:AllGrass} of $H^*_T({G}_{2k+1}(\R^{2n+2}))$, we will also get an element $\tilde{r}^T \in H^*_T(\tilde{G}_{2k+1}(\R^{2n+2}))$ in GKM description localized at a fixed circle $\tilde{C}_S=\{(\oplus_{i\in S} \R^2_{[\alpha_i]})\oplus L_0 \mid L_0\in \tilde{G}_1(\R^2_0)\}\cong S^1$ to be $\tilde{r}^T_S = (\prod_{i=1}^{n}\alpha_i) \theta_{S^1}$ similar to the ${r}^T \in H^*_T({G}_{2k+1}(\R^{2n+2}))$ localized at ${C}_S=\{(\oplus_{i\in S} \R^2_{[\alpha_i]})\oplus L_0 \mid L_0\in {G}_1(\R^2_0)\}\cong \R P^1$ to be ${r}^T_S = (\prod_{i=1}^{n}\alpha_i) \theta_{\R P^1}$, where $\theta_{S^1}$ and $\theta_{\R P^1}$ are the unit volume forms of $S^1$ and $\R P^1$ respectively. 

\begin{prop}[Canonical basis of $H^*_T(\tilde{G}_{2k+1}(\R^{2n+2}))$]\label{prop:CanoBaseIII}
	Let $\sigma_{S\in \mathcal{S}}$ be the canonical basis of $H^*_T({G}_{2k}(\R^{2n}))$ from Theorem \ref{thm:RealSchub} and $\tilde{r}^T_S = \prod_{i=1}^{n}\alpha_i \theta_{S^1}$ be the odd-degree generator. Then $\sigma_{S}, \tilde{r}^T \cdot \sigma_{S}$ give additive $\Q[\alpha_1,\ldots,\alpha_n]$-basis of $ H^*_T(\tilde{G}_{2k+1}(\R^{2n+2}))$.
\end{prop}

However, there is a subtlety for the pullback $\pi^*:H^*_T({G}_{2k+1}(\R^{2n+2}))\rightarrow H^*_T(\tilde{G}_{2k+1}(\R^{2n+2}))$ though this is an isomorphism. The $2$-fold covering $\pi: \tilde{G}_{2k+1}(\R^{2n+2}) \rightarrow {G}_{2k+1}(\R^{2n+2})$ restricts to a $2$-fold covering of fixed circles $\pi: (\tilde{C}_S\cong S^1) \rightarrow ({C}_S\cong \R P^1)$ which will give the localized pullback $\pi^*(\theta_{\R P^1})=2\theta_{S^1}$. Hence we get $\pi^*({r}^T)=2\tilde{r}^T$.

\begin{prop}[The explicit pullback of cohomology between odd dimensional Grassmannians]\label{prop:Pullr}
	In the canonical basis, the pullback of cohomology of odd dimensional Grassmannian is 
	\begin{align*}
	\pi^*:H^*_T({G}_{2k+1}(\R^{2n+2}))&\longrightarrow H^*_T(\tilde{G}_{2k+1}(\R^{2n+2}))\\
	\sigma_{S} &\longmapsto \sigma_{S}\\
	{r}^T \Cdot \sigma_{S} &\longmapsto 2\tilde{r}^T \Cdot \sigma_{S}.
	\end{align*}	
\end{prop}

For the even dimensional oriented Grassmannians, the deck transformation $\rho: \tilde{G}_k(\R^n)\rightarrow \tilde{G}_k(\R^n)$ switches any fixed point ${S_+}$ with its twin fixed point ${S_-}$ by reversing orientations. Then the induced deck transformation $\rho^*:H^*_T(\tilde{G}_k(\R^n))\rightarrow H^*_T(\tilde{G}_k(\R^n))$ in GKM description will switch any polynomial $f_{S_+}$ with $f_{S_-}$. Notice the symmetry in the GKM descriptions, we see that the switch of polynomials preserves the congruence relations.

Since $(\rho^*)^2=id$, both cohomology $H^*(\tilde{G}_k(\R^n)),H^*_T(\tilde{G}_k(\R^n))$ decompose
 into $\pm 1$-eigenspaces of $\rho^*$. 

\begin{prop}\label{prop:RhoEigen}
	For the even dimensional oriented Grassmannians $\tilde{G}_{2k}(\R^{2n}),\tilde{G}_{2k}(\R^{2n+1}),\tilde{G}_{2k+1}(\R^{2n+1})$, the elements of $+1$-eigenspace of $\rho^*$ on the equivariant cohomology can be identified as those sets of polynomials $\{f_{S_\pm}, S \in \mathcal{S}\}$ where $\mathcal{S}$ is the collection of $k$-element subsets of $\{1,\ldots,n\}$ such that
	\begin{align*}
	f_{S_+}=f_{S_-}
	\end{align*}
	and the elements of $-1$-eigenspace of $\rho^*$ are those with
	\[
	f_{S_+}=-f_{S_-}.
	\]
\end{prop}

\begin{rmk}
	As we have proved before, the $+1$-eigenspaces of $\rho^*$ on the equivariant cohomology of oriented Grassmannians are exactly the equivariant cohomology of real Grassmannians.
\end{rmk}

Recall that we defined equivariant Euler classes at each fixed point $S$ for $\tilde{G}_{2k}(\R^{2n}),\tilde{G}_{2k}(\R^{2n+1})$ to be $e^T_{S_{\pm}}=\pm \prod_{i\in S} \alpha_i$ and for $\tilde{G}_{2k}(\R^{2n}),\tilde{G}_{2k+1}(\R^{2n+1})$ to be $\bar{e}^T_{S_{\pm}}=\pm \prod_{j\not\in S} \alpha_j$. It is easy to check that $\{e^T_{S_\pm}, S \in \mathcal{S}\}$ and $\{\bar{e}^T_{S_\pm}, S \in \mathcal{S}\}$ are elements of the GKM description of the corresponding equivariant cohomology. Since $\rho$ changes the signs of orientations, $\rho^*$ changes the signs of the equivariant Euler classes. Therefore, $\{e^T_{S_\pm}, S \in \mathcal{S}\}$ and $\{\bar{e}^T_{S_\pm}, S \in \mathcal{S}\}$ are in the $-1$-eigenspaces of $\rho^*$. Topologically, the localized classes $e^T,\bar{e}^T$ in GKM description are exactly the equivariant Euler classes of the canonical oriented bundles and complementary oriented bundles over the oriented Grassmannians.

\begin{prop}[Equivariant Euler class and top equivariant Pontryagin class]\label{prop:EulerPont}
	Similar to the relations between ordinary Euler class and top ordinary Pontryagin class, 
	\begin{enumerate}
		\item For $\tilde{G}_{2k}(\R^{2n})$ and $\tilde{G}_{2k}(\R^{2n+1})$, we have $(e^T)^2 = p^T_k $
		\item For $\tilde{G}_{2k}(\R^{2n})$ and $\tilde{G}_{2k+1}(\R^{2n+1})$, we have $(\bar{e}^T)^2 = \bar{p}^T_{n-k}$
		\item For $\tilde{G}_{2k}(\R^{2n})$, we have $e^T\bar{e}^T= \prod_{i=1}^{n}\alpha_i$
	\end{enumerate}
\end{prop}
\begin{proof}
	Let's prove this for $\tilde{G}_{2k}(\R^{2n})$ which covers the remaining cases of $\tilde{G}_{2k}(\R^{2n+1}), \tilde{G}_{2k+1}(\R^{2n+1})$. In Proposition \ref{prop:Pont}, we have given the localized top equivariant Pontryagin classes of real Grassmannians as 
	\[
	p^T_k|_S=\prod_{i \in S} \alpha_i^2  \qquad \qquad \bar{p}^T_{n-k}|_S=\prod_{j \not \in S} \alpha_j^2.
	\]
	Via the pullback $\pi^*: H^*_T({G}_{2k}(\R^{2n})) \rightarrow H^*_T(\tilde{G}_{2k}(\R^{2n}))$, the equivariant Pontryagin classes of ${G}_{2k}(\R^{2n})$ are identified as those of $\tilde{G}_{2k}(\R^{2n})$, and are in the $+1$-eigenspaces of $\rho^*$. Therefore
	\[
	p^T_k|_{S_\pm}=\prod_{i \in S} \alpha_i^2  \qquad \qquad \bar{p}^T_{n-k}|_{S_\pm}=\prod_{j \not \in S} \alpha_j^2.
	\]
	Comparing them with
	\[
	e^T_{S_{\pm}}=\pm \prod_{i\in S} \alpha_i    \qquad \qquad \bar{e}^T_{S_{\pm}}=\pm \prod_{j\not\in S} \alpha_j
	\] 
	we get the stated relations.
\end{proof}

The induced deck transformation $\rho^*$ is a ring homomorphism, therefore the multiplication of an element in the $-1$-eigenspace with an element in the $+1$-eigenspace results in the $-1$-eigenspace. 

\begin{prop}\label{prop:EulerMult}
	Multiplication with the equivariant Euler classes $e^T,\bar{e}^T$ maps $+1$-eigenspaces of $\rho^*$ to $-1$-eigenspaces.
	\begin{enumerate}
		\item For $\tilde{G}_{2k}(\R^{2n+1})$, the multiplication with $e^T$ is an isomorphism between $+1$-eigenspace of $\rho^*$ to its $-1$-eigenspace.
		\item For $\tilde{G}_{2k+1}(\R^{2n+1})$, the multiplication with $\bar{e}^T$ is an isomorphism between $+1$-eigenspace of $\rho^*$ to its $-1$-eigenspace.
	\end{enumerate}
\end{prop}
\begin{proof}
	For $\tilde{G}_{2k}(\R^{2n+1})$, denote $V_{+1}$ and $V_{-1}$ be the $+1$ and $-1$-eigenspaces of $\rho^*$ on $H^*_T(\tilde{G}_{2k}(\R^{2n+1}))$. The fact that $e^T$ is in the $-1$-eigenspace gives the multiplication $\times e^T: V_{+1}\rightarrow V_{-1}$. On the other hand, every element $\{f_{S_\pm}, S \in \mathcal{S}\}$ of $V_{-1}$ has the form $f_{S_+}=-f_{S_-}$ by Prop \ref{prop:RhoEigen}. Plug this into the congruence relation between $S_+$ and $S_-$ in Theorem \ref{thm:GKMorientGrass}, we get 
	\[
	f_{S_+} \equiv f_{S_-}=-f_{S_+} \mod \prod_{i\in S}\alpha_{i}
	\]
	or equivalently, both $f_{S_+}$ and $f_{S_-}$ are multiples of $e^T_{S_{\pm}}=\pm \prod_{i\in S} \alpha_i$. Therefore, the localized quotients $f_{S_+}/e^T_{S_+}, f_{S_-}/e^T_{S_-}\in \Q[\alpha_1,\ldots,\alpha_n]$ are polynomials, and this defines a unique element $f/e^T \in V_{+1}$.
	
	The case of $\tilde{G}_{2k+1}(\R^{2n+1})$ is similar.
\end{proof}

\begin{rmk}
	For $\tilde{G}_{2k}(\R^{2n})$, neither the multiplication by $e^T$ nor by $\bar{e}^T$ are isomorphisms between the $+1$ and $-1$-eigenspaces of $\rho^*$. We will try to understand the equivariant cohomology of $\tilde{G}_{2k}(\R^{2n})$ in next subsection. 
\end{rmk}

The above isomorphism between eigenspaces of $\rho^*$, together with the canonical basis $\sigma_{S}$ of $H^*_T({G}_{2k}(\R^{2n+1}))$ and $H^*_T({G}_{2k+1}(\R^{2n+1}))$, give
\begin{prop}[Canonical basis of $H^*_T(\tilde{G}_{2k}(\R^{2n+1})),H^*_T(\tilde{G}_{2k+1}(\R^{2n+1}))$]\label{prop:CanoBaseII}
	Let $\sigma_{S\in \mathcal{S}}$ be the canonical basis of $H^*_T({G}_{2k}(\R^{2n+1}))$ and $H^*_T({G}_{2k+1}(\R^{2n+1}))$ from Theorem \ref{thm:RealSchub}. Then $\sigma_{S}, e^T \cdot \sigma_{S}$ and $\sigma_{S}, \bar{e}^T \cdot \sigma_{S}$ give additive $\Q[\alpha_1,\ldots,\alpha_n]$-basis of $ H^*_T(\tilde{G}_{2k}(\R^{2n+1}))$ and $H^*_T(\tilde{G}_{2k+1}(\R^{2n+1}))$ respectively.
\end{prop}

\begin{cor}
	The Poincar\'e series of $\tilde{G}_{2k}(\R^{2n+1})$ and $\tilde{G}_{2k+1}(\R^{2n+1})$ are
	\begin{align*}
	P_{\tilde{G}_{2k}(\R^{2n+1})}(t)&=(1+t^{2k})P_{{G}_{2k}(\R^{2n+1})}(t)=(1+t^{2k})P_{G_{k}(\C^{n})}(t^2)\\
	P_{\tilde{G}_{2k+1}(\R^{2n+1})}(t)&=(1+t^{2n-2k})P_{{G}_{2k+1}(\R^{2n+1})}(t)=(1+t^{2n-2k})P_{G_{k}(\C^{n})}(t^2).
	\end{align*}
\end{cor}

\begin{cor}[Relations between some oriented and real Grassmannians]\label{thm:Type2Grass}
	The equivariant cohomology of $\tilde{G}_{2k}(\R^{2n+1})$ and $\tilde{G}_{2k+1}(\R^{2n+1})$ are $\Q[\alpha_1,\ldots,\alpha_n]$-algebra extensions by $e^T,\bar{e}^T$ of the equivariant cohomology of ${G}_{2k}(\R^{2n+1})$ and ${G}_{2k+1}(\R^{2n+1})$, i.e.
	\begin{align*}
	H^*_T(\tilde{G}_{2k}(\R^{2n+1}))&\cong \frac{H^*_T({G}_{2k}(\R^{2n+1}))[e^T]}{(e^T)^2 = p^T_k} \\
	H^*_T(\tilde{G}_{2k+1}(\R^{2n+1}))&\cong \frac{H^*_T({G}_{2k+1}(\R^{2n+1}))[\bar{e}^T]}{(\bar{e}^T)^2 = \bar{p}^T_{n-k}}.
	\end{align*}
\end{cor}
\begin{proof}
	Using Prop \ref{prop:EulerPont}, \ref{prop:CanoBaseII} and dimension counting.
\end{proof}

\subsection{Leray-Borel description of oriented Grassmannians}
In this subsection, we will confirm the ring generators of equivariant cohomology of oriented Grassmannians to be characteristic classes, then determine the complete relations among them, and also give additive basis.

\subsubsection{Leray-Borel description of $\tilde{G}_{2k}(\R^{2n+1}),\tilde{G}_{2k+1}(\R^{2n+1}),\tilde{G}_{2k+1}(\R^{2n+2})$}
From Theorem \ref{thm:OddGrass} and Theorem \ref{thm:Type2Grass}, we have seen that equivariant cohomology rings of $\tilde{G}_{2k}(\R^{2n+1})$, $\tilde{G}_{2k+1}(\R^{2n+1})$ and $\tilde{G}_{2k+1}(\R^{2n+2})$ are ring extensions of the equivariant cohomology of their real counterparts. Hence the equivariant Leray-Borel descriptions and equivariant characteristic basis of those oriented Grassmannians can be extended from the related real Grassmannians.

\begin{thm}[Equivariant Leray-Borel description of $\tilde{G}_{2k}(\R^{2n+1}),\tilde{G}_{2k+1}(\R^{2n+1}),\tilde{G}_{2k+1}(\R^{2n+2})$]
	The equivariant cohomology rings of $\tilde{G}_{2k}(\R^{2n+1}),\tilde{G}_{2k+1}(\R^{2n+1}),\tilde{G}_{2k+1}(\R^{2n+2})$ are generated by equivariant Pontryagin and Euler classes, and an odd-degree class $\tilde{r}^T$:
	\begin{align*}
	H^*_T(\tilde{G}_{2k}(\R^{2n+1})) &\cong \frac{\Q[\alpha_1,\alpha_2,\ldots,\alpha_n][p^T_1,p^T_2,\ldots,p^T_k;\bar{p}^T_1,\bar{p}^T_2,\ldots,\bar{p}^T_{n-k};e^T]}{p^T\bar{p}^T = \prod_{i=1}^{n}(1+\alpha^2_i),\, (e^T)^2=p^T_k}\\
	H^*_T(\tilde{G}_{2k+1}(\R^{2n+1})) &\cong \frac{\Q[\alpha_1,\alpha_2,\ldots,\alpha_n][p^T_1,p^T_2,\ldots,p^T_k;\bar{p}^T_1,\bar{p}^T_2,\ldots,\bar{p}^T_{n-k};\bar{e}^T]}{p^T\bar{p}^T = \prod_{i=1}^{n}(1+\alpha^2_i),\, (\bar{e}^T)^2=\bar{p}^T_{n-k}}\\
	H^*_T(\tilde{G}_{2k+1}(\R^{2n+2}))&\cong\frac{\Q[\alpha_1,\alpha_2,\ldots,\alpha_n][p^T_1,p^T_2,\ldots,p^T_k;\bar{p}^T_1,\bar{p}^T_2,\ldots,\bar{p}^T_{n-k};\tilde{r}^T]}{p^T\bar{p}^T = \prod_{i=1}^{n}(1+\alpha^2_i),\,(\tilde{r}^T)^2=0}.
	\end{align*}
\end{thm}

\begin{thm}[Equivariant characteristic basis of $\tilde{G}_{2k}(\R^{2n+1}),\tilde{G}_{2k+1}(\R^{2n+1}),\tilde{G}_{2k+1}(\R^{2n+2})$]\label{thm:EquivCharOrient}
	The sets of monomials $\{(p^T_1)^{r_1}(p^T_2)^{r_2}\cdots (p^T_k)^{r_k},\,e^T\cdot(p^T_1)^{r_1}(p^T_2)^{r_2}\cdots (p^T_k)^{r_k}\}$, $\{(p^T_1)^{r_1}(p^T_2)^{r_2}\cdots (p^T_k)^{r_k},\,\bar{e}^T\cdot(p^T_1)^{r_1}(p^T_2)^{r_2}\cdots (p^T_k)^{r_k}\}$ and $\{(p^T_1)^{r_1}(p^T_2)^{r_2}\cdots (p^T_k)^{r_k},\,\tilde{r}^T\cdot(p^T_1)^{r_1}(p^T_2)^{r_2}\cdots (p^T_k)^{r_k}\}$ satisfying the condition $\sum_{i=1}^{k} r_i \leq n-k$ form additive $H^*_T(pt)$-basis for $H^*_T(\tilde{G}_{2k}(\R^{2n+1})),H^*_T(\tilde{G}_{2k+1}(\R^{2n+1})),H^*_T(\tilde{G}_{2k+1}(\R^{2n+2}))$ respectively. 
\end{thm} 

The above two theorems of equivariant ring generators and equivariant additive basis both have their ordinary versions by replacing $\alpha_i$ with $0$.

\begin{cor}[Ordinary Leray-Borel description of $\tilde{G}_{2k}(\R^{2n+1}),\tilde{G}_{2k+1}(\R^{2n+1}),\tilde{G}_{2k+1}(\R^{2n+2})$]
	The ordinary cohomology rings of $\tilde{G}_{2k}(\R^{2n+1}),\tilde{G}_{2k+1}(\R^{2n+1}),\tilde{G}_{2k+1}(\R^{2n+2})$ are generated by Pontryagin and Euler classes, and an odd-degree class $\tilde{r}$:
	\begin{align*}
	H^*(\tilde{G}_{2k}(\R^{2n+1})) &\cong \frac{\Q[p_1,p_2,\ldots,p_k;\bar{p}_1,\bar{p}_2,\ldots,\bar{p}_{n-k};e]}{p\bar{p} = 1,\, e^2=p_k}\\
	H^*(\tilde{G}_{2k+1}(\R^{2n+1})) &\cong \frac{\Q[p_1,p_2,\ldots,p_k;\bar{p}_1,\bar{p}_2,\ldots,\bar{p}_{n-k};\bar{e}]}{p\bar{p} = 1,\, \bar{e}^2=\bar{p}_{n-k}}\\
	H^*(\tilde{G}_{2k+1}(\R^{2n+2}))&\cong\frac{\Q[p_1,p_2,\ldots,p_k;\bar{p}_1,\bar{p}_2,\ldots,\bar{p}_{n-k};\tilde{r}]}{p\bar{p} = 1,\, \tilde{r}^2=0}.
	\end{align*}
\end{cor}

\begin{cor}[Ordinary characteristic basis of $\tilde{G}_{2k}(\R^{2n+1}),\tilde{G}_{2k+1}(\R^{2n+1}),\tilde{G}_{2k+1}(\R^{2n+2})$]
	The sets of monomials $\{p_1^{r_1}p_2^{r_2}\cdots p_k^{r_k},\,e\cdot p_1^{r_1}p_2^{r_2}\cdots p_k^{r_k}\}$, $\{p_1^{r_1}p_2^{r_2}\cdots p_k^{r_k},\,\bar{e}\cdot p_1^{r_1}p_2^{r_2}\cdots p_k^{r_k}\}$ and $\{p_1^{r_1}p_2^{r_2}\cdots p_k^{r_k},\,\tilde{r}\cdot p_1^{r_1}p_2^{r_2}\cdots p_k^{r_k}\}$ satisfying the condition $\sum_{i=1}^{k} r_i \leq n-k$ form additive basis for $H^*(\tilde{G}_{2k}(\R^{2n+1})),H^*(\tilde{G}_{2k+1}(\R^{2n+1})),H^*(\tilde{G}_{2k+1}(\R^{2n+2}))$ respectively. 
\end{cor} 

\subsubsection{Leray-Borel description of $\tilde{G}_{2k}(\R^{2n})$}
Now let's turn to the remaining type of oriented Grassmannian $\tilde{G}_{2k}(\R^{2n})$. As we remarked in previous subsection, neither the multiplication by $e^T$ nor by $\bar{e}^T$ are isomorphisms between eigenspaces of $\rho^*$. However, we will show the multiplications by $e^T$ and $\bar{e}^T$, restricted on certain carefully chosen subspaces, do give isomorphism between $+1$ and $-1$-eigenspaces of $\rho^*$. 

Notice the equivariant diffeomorphism ${G}_{2k}(\R^{2n}) \cong {G}_{2n-2k}(\R^{2n})$ by mapping an oriented $2k$-dimensional subspace to its perpendicular oriented $(2n-2k)$-dimensional subspace. Then the complementary characteristic monomials $(\bar{p}^T_1)^{r_1}(\bar{p}^T_2)^{r_2}\cdots (\bar{p}^T_{n-k})^{r_{n-k}}$ and $\bar{p}_1^{r_1}\bar{p}_2^{r_2}\cdots \bar{p}_{n-k}^{r_{n-k}}$, satisfying the condition $\sum_{i=1}^{n-k} r_i \leq k$, give additive basis for the equivariant and respectively ordinary cohomology of ${G}_{2k}(\R^{2n})$. Also recall from Prop \ref{prop:EulerPont} on the relations among top Pontryagin classes and Euler classes of the oriented Grassmannian $\tilde{G}_{2k}(\R^{2n})$ that $(e^T)^2=p^T_k,(\bar{e}^T)^2=\bar{p}^T_{n-k},e^T\bar{e}^T=\prod_{i=1}^{n}\alpha_i$ and $e^2=p_k,\bar{e}^2=\bar{p}_{n-k},e\bar{e}=0$.

\begin{prop}[Eigenspaces of $H^*(\tilde{G}_{2k}(\R^{2n}))$]\label{prop:OrientEigen}
	Let $\rho$ be the non-trivial deck transformation of the covering $\pi: \tilde{G}_{2k}(\R^{2n})\rightarrow {G}_{2k}(\R^{2n})$ and identify $H^*({G}_{2k}(\R^{2n}))$ as the $+1$-eigenspace of $\rho^*$ on $H^*(\tilde{G}_{2k}(\R^{2n}))$. 
	\begin{enumerate}
		\item The multiplications by $e^T$ and $\bar{e}^T$ are isomorphisms restricted on the following subspaces
		\begin{align*}
		e \times: \mathrm{Span}(p_1^{r_1}p_2^{r_2}\cdots p_k^{r_k}\mid \sum_{i=1}^{k} r_i \leq n-k-1) & \overset{\cong}{\longrightarrow}
		e\cdot \mathrm{Span}(p_1^{r_1}p_2^{r_2}\cdots p_k^{r_k}\mid \sum_{i=1}^{k} r_i \leq n-k-1)\\
        \bar{e} \times: \mathrm{Span}(\bar{p}_1^{r_1}\bar{p}_2^{r_2}\cdots \bar{p}_{n-k}^{r_{n-k}}\mid \sum_{i=1}^{n-k} r_i \leq k-1) & \overset{\cong}{\longrightarrow}
        \bar{e}\cdot \mathrm{Span}(\bar{p}_1^{r_1}\bar{p}_2^{r_2}\cdots \bar{p}_{n-k}^{r_{n-k}}\mid \sum_{i=1}^{n-k} r_i \leq k-1).
		\end{align*}
		\item $e\cdot \mathrm{Span}(p_1^{r_1}p_2^{r_2}\cdots p_k^{r_k}\mid \sum_{i=1}^{k} r_i \leq n-k-1) \oplus \bar{e}\cdot \mathrm{Span}(\bar{p}_1^{r_1}\bar{p}_2^{r_2}\cdots \bar{p}_{n-k}^{r_{n-k}}\mid \sum_{i=1}^{n-k} r_i \leq k-1)$ is the $(-1)$-eigenspace of $\rho^*$
		\item $e\cdot H^*({G}_{2k}(\R^{2n})) \cap \bar{e} \cdot H^*({G}_{2k}(\R^{2n}))=0$ and $e\cdot H^*({G}_{2k}(\R^{2n})) \oplus \bar{e} \cdot H^*({G}_{2k}(\R^{2n}))$ is the $(-1)$-eigenspace of $\rho^*$
		\item The kernels of $e \times$ and $\bar{e} \times$ on $H^*({G}_{2k}(\R^{2n}))$ are $\bar{p}_{n-k}\cdot \mathrm{Span}(\bar{p}_1^{r_1}\bar{p}_2^{r_2}\cdots \bar{p}_{n-k}^{r_{n-k}}\mid \sum_{i=1}^{n-k} r_i \leq k-1)$ and $p_k \cdot \mathrm{Span}(p_1^{r_1}p_2^{r_2}\cdots p_k^{r_k}\mid \sum_{i=1}^{k} r_i \leq n-k-1)$ respectively
		\item The following spaces are identical
		\begin{align*}
		e\cdot \mathrm{Span}(p_1^{r_1}p_2^{r_2}\cdots p_k^{r_k}\mid \sum_{i=1}^{k} r_i \leq n-k-1) = 
		e\cdot \mathrm{Span}(\bar{p}_1^{r_1}\bar{p}_2^{r_2}\cdots \bar{p}_{n-k-1}^{r_{n-k-1}}\mid \sum_{i=1}^{n-k-1} r_i \leq k) = e\cdot H^*({G}_{2k}(\R^{2n}))\\
		\bar{e}\cdot \mathrm{Span}(\bar{p}_1^{r_1}\bar{p}_2^{r_2}\cdots \bar{p}_{n-k}^{r_{n-k}}\mid \sum_{i=1}^{n-k} r_i \leq k-1) =
		\bar{e}\cdot \mathrm{Span}(p_1^{r_1}p_2^{r_2}\cdots p_{k-1}^{r_{k-1}}\mid \sum_{i=1}^{k-1} r_i \leq n-k) =
		\bar{e} \cdot H^*({G}_{2k}(\R^{2n})). 
		\end{align*}
	\end{enumerate}
\end{prop}
\begin{proof}
	Note that the total Betti numbers of $H^*({G}_{2k}(\R^{2n}))$ and $H^*(\tilde{G}_{2k}(\R^{2n}))$ are $\binom{n}{k}$ and $2\binom{n}{k}$ respectively, hence the dimension of the $-1$-eigenspace of $\rho^*$ is $\binom{n}{k}$.
	\begin{enumerate}
		\item The composition of the surjective linear maps
		\begin{align*}
		&e \times: \mathrm{Span}(p_1^{r_1}p_2^{r_2}\cdots p_k^{r_k}\mid \sum_{i=1}^{k} r_i \leq n-k-1)  \longrightarrow
		e\cdot \mathrm{Span}(p_1^{r_1}p_2^{r_2}\cdots p_k^{r_k}\mid \sum_{i=1}^{k} r_i \leq n-k-1)\\
		&e \times: e\cdot \mathrm{Span}(p_1^{r_1}p_2^{r_2}\cdots p_k^{r_k}\mid \sum_{i=1}^{k} r_i \leq n-k-1)  \longrightarrow e^2\cdot \mathrm{Span}(p_1^{r_1}p_2^{r_2}\cdots p_k^{r_k}\mid \sum_{i=1}^{k} r_i \leq n-k-1)
		\end{align*}
		is 
		\[
		p_k\times: \mathrm{Span}(p_1^{r_1}p_2^{r_2}\cdots p_k^{r_k}\mid \sum_{i=1}^{k} r_i \leq n-k-1)  \longrightarrow
		p_k\cdot \mathrm{Span}(p_1^{r_1}p_2^{r_2}\cdots p_k^{r_k}\mid \sum_{i=1}^{k} r_i \leq n-k-1)
		\]
		where we have used the relation $e^2=p_k$. The composition maps a sub-basis of $H^*({G}_{2k}(\R^{2n}))$ onto another sub-basis without common vectors, hence is a bijection. Therefore, each individual surjection is a bijection. Similarly, we get the bijection for the restricted $\bar{e} \times$.
		\item
		We have seen from the above that
		\[
		e \times: e\cdot \mathrm{Span}(p_1^{r_1}p_2^{r_2}\cdots p_k^{r_k}\mid \sum_{i=1}^{k} r_i \leq n-k-1) \longrightarrow p_k\cdot \mathrm{Span}(p_1^{r_1}p_2^{r_2}\cdots p_k^{r_k}\mid \sum_{i=1}^{k} r_i \leq n-k-1)
		\]   
		is a bijection. However, $e \times$ takes $\bar{e} \cdot H^*({G}_{2k}(\R^{2n}))$ to zero, because $e\bar{e}=0$. Hence 
		\[
		e\cdot \mathrm{Span}(p_1^{r_1}p_2^{r_2}\cdots p_k^{r_k}\mid \sum_{i=1}^{k} r_i \leq n-k-1) \cap \bar{e} \cdot H^*({G}_{2k}(\R^{2n})) = 0.
		\]
		Similarly,
		\[
		\bar{e}\cdot \mathrm{Span}(\bar{p}_1^{r_1}\bar{p}_2^{r_2}\cdots \bar{p}_{n-k}^{r_{n-k}}\mid \sum_{i=1}^{n-k} r_i \leq k-1) \cap e\cdot H^*({G}_{2k}(\R^{2n})) = 0.
		\]
		Combine these two, we get
		\[
		e\cdot \mathrm{Span}(p_1^{r_1}p_2^{r_2}\cdots p_k^{r_k}\mid \sum_{i=1}^{k} r_i \leq n-k-1) \cap \bar{e}\cdot \mathrm{Span}(\bar{p}_1^{r_1}\bar{p}_2^{r_2}\cdots \bar{p}_{n-k}^{r_{n-k}}\mid \sum_{i=1}^{n-k} r_i \leq k-1) = 0.
		\]
		However, as a subspace in $-1$-eigenspace of $\rho^*$, the sum $e\cdot\mathrm{Span}(p_1^{r_1}p_2^{r_2}\cdots p_k^{r_k}\mid \sum_{i=1}^{k} r_i \leq n-k-1) \oplus \bar{e}\cdot \mathrm{Span}(\bar{p}_1^{r_1}\bar{p}_2^{r_2}\cdots \bar{p}_{n-k}^{r_{n-k}}\mid \sum_{i=1}^{n-k} r_i \leq k-1)$ has dimension $\binom{n-1}{k}+\binom{n-1}{n-k}=\binom{n}{k}$ the same as dimension of the entire $-1$-eigenspace of $\rho^*$, hence is exactly the $-1$-eigenspace of $\rho^*$. 
		\item
		The above series of zero intersections force $e\cdot\mathrm{Span}(p_1^{r_1}p_2^{r_2}\cdots p_k^{r_k}\mid \sum_{i=1}^{k} r_i \leq n-k-1)=e\cdot H^*({G}_{2k}(\R^{2n}))$ and $\bar{e}\cdot \mathrm{Span}(\bar{p}_1^{r_1}\bar{p}_2^{r_2}\cdots \bar{p}_{n-k}^{r_{n-k}}\mid \sum_{i=1}^{n-k} r_i \leq k-1)=\bar{e}\cdot H^*({G}_{2k}(\R^{2n}))$. Hence we get $e\cdot H^*({G}_{2k}(\R^{2n})) \cap \bar{e} \cdot H^*({G}_{2k}(\R^{2n}))=0$ and $e\cdot H^*({G}_{2k}(\R^{2n})) \oplus \bar{e} \cdot H^*({G}_{2k}(\R^{2n}))$ is the $(-1)$-eigenspace of $\rho^*$.
		\item
		We have proved $e\cdot H^*({G}_{2k}(\R^{2n}))=e\cdot\mathrm{Span}(p_1^{r_1}p_2^{r_2}\cdots p_k^{r_k}\mid \sum_{i=1}^{k} r_i \leq n-k-1)$ and they are of dimension $\binom{n-1}{k}$. Since $H^*({G}_{2k}(\R^{2n}))$ is of dimension $\binom{n}{k}$, the kernel of $e\times$ on $H^*({G}_{2k}(\R^{2n}))$ is then of dimension $\binom{n}{k}-\binom{n-1}{k}=\binom{n-1}{n-k}$. Because $e\cdot \bar{p}_{n-k}=e\cdot \bar{e}^2=0$, the subspace $\bar{p}_{n-k}\cdot \mathrm{Span}(\bar{p}_1^{r_1}\bar{p}_2^{r_2}\cdots \bar{p}_{n-k}^{r_{n-k}}\mid \sum_{i=1}^{n-k} r_i \leq k-1)$ of dimension $\binom{n-1}{n-k}$ is clearly in the kernel of $e\times$ on $H^*({G}_{2k}(\R^{2n}))$, hence is exactly the kernel. Similarly, we obtain the kernel of $\bar{e}\times$ on $H^*({G}_{2k}(\R^{2n}))$.
		\item
		The $e\times$-kernel subspace $\bar{p}_{n-k}\cdot \mathrm{Span}(\bar{p}_1^{r_1}\bar{p}_2^{r_2}\cdots \bar{p}_{n-k}^{r_{n-k}}\mid \sum_{i=1}^{n-k} r_i \leq k-1)$ of $H^*({G}_{2k}(\R^{2n}))$ has complementary subspace $\mathrm{Span}(\bar{p}_1^{r_1}\bar{p}_2^{r_2}\cdots \bar{p}_{n-k-1}^{r_{n-k-1}}\mid \sum_{i=1}^{n-k-1} r_i \leq k)$. Hence the restriction 
		\[
		e\times: \mathrm{Span}(\bar{p}_1^{r_1}\bar{p}_2^{r_2}\cdots \bar{p}_{n-k-1}^{r_{n-k-1}}\mid \sum_{i=1}^{n-k-1} r_i \leq k) \longrightarrow e\cdot H^*({G}_{2k}(\R^{2n}))
		\]
		is bijection, therefore $e\cdot \mathrm{Span}(\bar{p}_1^{r_1}\bar{p}_2^{r_2}\cdots \bar{p}_{n-k-1}^{r_{n-k-1}}\mid \sum_{i=1}^{n-k-1} r_i \leq k) = e\cdot H^*({G}_{2k}(\R^{2n}))$. The identification $e\cdot\mathrm{Span}(p_1^{r_1}p_2^{r_2}\cdots p_k^{r_k}\mid \sum_{i=1}^{k} r_i \leq n-k-1)=e\cdot H^*({G}_{2k}(\R^{2n}))$ is proved in (3). Similarly, we get the identifications for $\bar{e}\cdot H^*({G}_{2k}(\R^{2n}))$.
	\end{enumerate}
\end{proof}

The detailed discussion of $e\times$ and $\bar{e}\times$ between the eigenspaces of $\rho^*$ gives:
\begin{cor}[Ordinary characteristic basis of $\tilde{G}_{2k}(\R^{2n})$]
	The ordinary cohomology of $\tilde{G}_{2k}(\R^{2n})$ is generated by Pontryagin classes and Euler classes with an additive basis $\{p_1^{r_1}p_2^{r_2}\cdots p_k^{r_k}\mid \sum_{i=1}^{k} r_i \leq n-k\}$ for the $+1$-eigenspace of $\rho^*$ and $\{e\cdot\bar{p}_1^{r_1}\bar{p}_2^{r_2}\cdots \bar{p}_{n-k-1}^{r_{n-k-1}}\mid \sum_{i=1}^{n-k-1} r_i \leq k\}\cup \{ \bar{e}\cdot p_1^{r_1}p_2^{r_2}\cdots p_{k-1}^{r_{k-1}}\mid \sum_{i=1}^{k-1} r_i \leq n-k\}$ for the $-1$-eigenspace.
\end{cor}

\begin{rmk}
	Using the various identifications of $e\cdot H^*({G}_{2k}(\R^{2n}))$ and $\bar{e}\cdot H^*({G}_{2k}(\R^{2n}))$ in Theorem \ref{prop:OrientEigen}, we can also give the additive basis of $\tilde{G}_{2k}(\R^{2n})$ in other forms.
\end{rmk}

\begin{cor}
	The Poincar\'e series of $\tilde{G}_{2k}(\R^{2n})$ are
	\begin{align*}
	P_{\tilde{G}_{2k}(\R^{2n})}(t)
	&=P_{{G}_{2k}(\R^{2n})}(t)+t^{2k}P_{{G}_{2k}(\R^{2n-2})}(t)+t^{2n-2k}P_{{G}_{2k-2}(\R^{2n-2})}(t)\\
	&=P_{G_{k}(\C^{n})}(t^2)+t^{2k}P_{G_{k}(\C^{n-1})}(t^2)+t^{2n-2k}P_{G_{k-1}(\C^{n-1})}(t^2).
	\end{align*} 
\end{cor}
\begin{proof}
	Notice that the $-1$-eigenbasis $\{e\cdot\bar{p}_1^{r_1}\bar{p}_2^{r_2}\cdots \bar{p}_{n-k-1}^{r_{n-k-1}}\mid \sum_{i=1}^{n-k-1} r_i \leq k\}\cup \{ \bar{e}\cdot p_1^{r_1}p_2^{r_2}\cdots p_{k-1}^{r_{k-1}}\mid \sum_{i=1}^{k-1} r_i \leq n-k\}$ has factors $\{\bar{p}_1^{r_1}\bar{p}_2^{r_2}\cdots \bar{p}_{n-k-1}^{r_{n-k-1}}\mid \sum_{i=1}^{n-k-1} r_i \leq k\}$ and $ \{p_1^{r_1}p_2^{r_2}\cdots p_{k-1}^{r_{k-1}}\mid \sum_{i=1}^{k-1} r_i \leq n-k\}$ which also appear as the additive basis of $H^*({G}_{2k}(\R^{2n-2}))$ and $H^*({G}_{2k-2}(\R^{2n-2}))$ respectively.
\end{proof}

\begin{rmk}
	The Poincar\'e series of even dimensional oriented Grassmannians $\tilde{G}_{2k}(\R^{2n}), \tilde{G}_{2k}(\R^{2n+1}),\tilde{G}_{2k+1}(\R^{2n+1})$ were already computed by H. Cartan \cite{Car50}.
\end{rmk}

\begin{thm}[Ordinary Leray-Borel description of $\tilde{G}_{2k}(\R^{2n})$]
	The ordinary cohomology of $\tilde{G}_{2k}(\R^{2n})$ is a ring extension of the ordinary cohomology of ${G}_{2k}(\R^{2n})$:
	\[
	H^*(\tilde{G}_{2k}(\R^{2n}))\cong \frac{H^*({G}_{2k}(\R^{2n}))[e,\bar{e}]}{e^2=p_k,\,\bar{e}^2=\bar{p}_{n-k},\,e\bar{e}=0} \cong \frac{\Q[p_1,p_2,\ldots,p_k;\bar{p}_1,\bar{p}_2,\ldots,\bar{p}_{n-k};e,\bar{e}]}{p\bar{p} = 1,\, e^2=p_k,\,\bar{e}^2=\bar{p}_{n-k},\,e\bar{e}=0}.
	\]
\end{thm}
\begin{proof}
	Consider the ring homomorphism
	\[
	\frac{H^*({G}_{2k}(\R^{2n}))[e,\bar{e}]}{e^2=p_k,\,\bar{e}^2=\bar{p}_{n-k},\,e\bar{e}=0} \longrightarrow H^*(\tilde{G}_{2k}(\R^{2n}))
	\]
	which sends Pontryagin classes of ${G}_{2k}(\R^{2n})$ to the corresponding Pontryagin classes of $\tilde{G}_{2k}(\R^{2n})$ and sends the abstract symbols $e,\bar{e}$ to the actual Euler classes of the oriented canonical bundle and complementary bundle over $\tilde{G}_{2k}(\R^{2n})$. Since we have proved that $H^*(\tilde{G}_{2k}(\R^{2n}))$ is generated by Pontryagin classes and Euler classes, the above morphism is surjective. It is easy check that $H^*({G}_{2k}(\R^{2n}))[e,\bar{e}]/\{e^2=p_k,\,\bar{e}^2=\bar{p}_{n-k},\,e\bar{e}=0\}$ also has the same additive basis $\{p_1^{r_1}p_2^{r_2}\cdots p_k^{r_k}\mid \sum_{i=1}^{k} r_i \leq n-k\}\cup\{e\cdot\bar{p}_1^{r_1}\bar{p}_2^{r_2}\cdots \bar{p}_{n-k-1}^{r_{n-k-1}}\mid \sum_{i=1}^{n-k-1} r_i \leq k\}\cup \{ \bar{e}\cdot p_1^{r_1}p_2^{r_2}\cdots p_{k-1}^{r_{k-1}}\mid \sum_{i=1}^{k-1} r_i \leq n-k\}$ as $H^*(\tilde{G}_{2k}(\R^{2n}))$. Hence, we get a ring isomorphism.
\end{proof}

Since $T^n$ acts on $\tilde{G}_{2k}(\R^{2n})$ equivariantly formal, i.e. $H_T^*(\tilde{G}_{2k}(\R^{2n}))\cong\Q[\alpha_1,\dots,\alpha_n]\otimes_\Q H^*(\tilde{G}_{2k}(\R^{2n}))$ as $\Q[\alpha_1,\dots,\alpha_n]$-modules, we can lift the ordinary basis, characteristic classes and relations to be equivariant, then obtain the equivariant versions of characteristic basis and Leray-Borel description:

\begin{cor}[Equivariant Leray-Borel description and characteristic basis of $\tilde{G}_{2k}(\R^{2n})$]
	The equivariant cohomology of $\tilde{G}_{2k}(\R^{2n})$ is a ring extension of the equivariant cohomology of ${G}_{2k}(\R^{2n})$:
	\begin{align*}
	H_T^*(\tilde{G}_{2k}(\R^{2n}))
	&\cong \frac{H_T^*({G}_{2k}(\R^{2n}))[e^T,\bar{e}^T]}{(e^T)^2=p^T_k,\,(\bar{e}^T)^2=\bar{p}^T_{n-k},\,e^T\bar{e}^T=\prod_{i=1}^{n}\alpha_i} \\
	&\cong \frac{\Q[\alpha_1,\alpha_2,\ldots,\alpha_n][p^T_1,p^T_2,\ldots,p^T_k;\bar{p}^T_1,\bar{p}^T_2,\ldots,\bar{p}^T_{n-k};e^T,\bar{e}^T]}{p^T\bar{p}^T = \prod_{i=1}^{n}(1+\alpha^2_i),\,(e^T)^2=p^T_k,\,(\bar{e}^T)^2=\bar{p}^T_{n-k},\,e^T\bar{e}^T=\prod_{i=1}^{n}\alpha_i}
	\end{align*}
	with additive $\Q[\alpha_1,\alpha_2,\ldots,\alpha_n]$-basis $\{(p^T_1)^{r_1}(p^T_2)^{r_2}\cdots (p^T_k)^{r_k}\mid \sum_{i=1}^{k} r_i \leq n-k\}$ for the $+1$-eigenspace of $\rho^*$ and $\{e^T\cdot(\bar{p}^T)_1^{r_1}(\bar{p}^T)_2^{r_2}\cdots (\bar{p}^T)_{n-k-1}^{r_{n-k-1}}\mid \sum_{i=1}^{n-k-1} r_i \leq k\}\cup \{ \bar{e}^T\cdot (p_1^T)^{r_1}(p_2^T)^{r_2}\cdots (p_{k-1}^T)^{r_{k-1}}\mid \sum_{i=1}^{k-1} r_i \leq n-k\}$ for the $-1$-eigenspace.
\end{cor}

\begin{rmk}
	The ordinary cohomology groups of $\tilde{G}_k(\R^n)$ in $\Z$ coefficients for $n\leq 8$ were computed by Jungkind \cite{Ju79}. The ordinary cohomology rings of $\tilde{G}_k(\R^n)$ in $\R$ coefficients for $k=2$ were computed by Shi\&Zhou \cite{SZ14}.
\end{rmk}

\subsubsection{Characteristic numbers of orientable Grassmannians}
All the oriented Grassmannians are canonically oriented. Among the real Grassmannians, only $G_{2k}(\R^{2n})$ and $G_{2k+1}(\R^{2n+2})$ have nonzero top Betti numbers and hence are orientable. We can integrate equivariant cohomology classes on these Grassmannians using the Atiyah-Bott-Berline-Vergne(ABBV) localization formula \ref{ABBV}. According to the additive $\Q[\alpha_1,\ldots,\alpha_n]$-module structures of equivariant cohomology of theses Grassmannians, we shall need to understand the integration of equivariant characteristic classes in various cases for any multi-index $I=(i_1,\ldots,i_k)$:  $(p^T)^I,\,e^T\cdot(p^T)^I,\,\bar{e}^T\cdot(p^T)^I,\,r^T\cdot(p^T)^I,\,\tilde{r}^T\cdot(p^T)^I$.

The equivariant Pontryagin classes of canonical bundles, complementary bundles and tangent bundles are given in Prop \ref{prop:Pont}. The equivariant Euler classes of canonical bundles and complementary bundles are given in Subsubsection \ref{subsub:Euler}. The $r^T,\tilde{r}^T$ are given in Theorem \ref{thm:AllGrass} and Prop \ref{prop:CanoBaseIII}. In order to apply the ABBV formula, we need a localized expression for the equivariant Euler class of normal bundle at each fixed point or fixed circle.

\begin{prop}
	Let $S$ be a $k$-element subset of $\{1,\ldots,n\}$, the equivariant Euler class of normal bundle at a fixed point or fixed circle associated to $S,S_\pm$ is
	\begin{enumerate}
		\item For $G_{2k}(\R^{2n}),\tilde{G}_{2k}(\R^{2n})$, 
		\[
		e^N_{S_\pm} = e^N_S = \prod_{i\in S}\prod_{j \not \in S}(\alpha^2_j-\alpha^2_i).
		\]
		\item For $\tilde{G}_{2k}(\R^{2n+1})$,
		\[
		e^N_{S_\pm} = \pm \prod_{l\in S}\alpha_l \prod_{i\in S}\prod_{j \not \in S}(\alpha^2_j-\alpha^2_i).
		\]
		\item For $\tilde{G}_{2k+1}(\R^{2n+1})$,
		\[
		e^N_{S_\pm} = \pm \prod_{l \not\in S}\alpha_l \prod_{i\in S}\prod_{j \not \in S}(\alpha^2_j-\alpha^2_i).
		\]
		\item For ${G}_{2k+1}(\R^{2n+1}),\tilde{G}_{2k+1}(\R^{2n+1})$,
		\[
		e^N_S = \prod_{l=1}^n\alpha_l \prod_{i\in S}\prod_{j \not \in S}(\alpha^2_j-\alpha^2_i).
		\]
	\end{enumerate}
\end{prop}
\begin{proof}
	The tangent spaces with weight decomposition at each fixed point or fixed circle of real Grassmannians (hence also oriented Grassmannians) are given in Subsubsection \ref{subsubsec:IsoWeights}, therefore we get the equivariant Euler classes of normal bundles up to signs as the expressions claimed in current Proposition. To resolve the sign ambiguity, we just need to note that the claimed expressions are invariant under the Weyl groups of the oriented Grassmannians as homogeneous spaces $G/H$, and also invariant under the deck transformation $\rho^*$. 
\end{proof}

Next, we will compute and relate equivariant characteristic numbers of different Grassmannians.

\begin{thm}[Equivariant characteristic numbers of real\&oriented Grassmannians]
	Let $I=(i_1,\ldots,i_k)$ be a multi-index and $\mathcal{S}$ be the collection of all $k$-element subsets of $\{1,\ldots,n\}$, then 
	\begin{align*}
	&\int_{\tilde{G}_{2k}(\R^{2n})}(p^T)^I=\int_{\tilde{G}_{2k}(\R^{2n+1})}e^T \cdot (p^T)^I=\int_{\tilde{G}_{2k+1}(\R^{2n+1})}\bar{e}^T \cdot (p^T)^I\\
	=&2\int_{{G}_{2k}(\R^{2n})}(p^T)^I=2\int_{{G}_{2k+1}(\R^{2n+2})}r^T\cdot (p^T)^I=2\int_{\tilde{G}_{2k+1}(\R^{2n+2})}\tilde{r}^T\cdot (p^T)^I\\
	=& 2 \sum_{S \in \mathcal{S}}\frac{\big((p_1^T)^{i_1}\cdots(p_k^T)^{i_k}\big)|_S}{\prod_{i\in S}\prod_{j \not \in S}(\alpha^2_j-\alpha^2_i)}.
	\end{align*}
\end{thm}
\begin{proof}
	When applying the ABBV localization formula \ref{ABBV}, besides the localized Pontryagin classes, we just need to observe that for ${G}_{2k}(\R^{2n})$, $\tilde{G}_{2k}(\R^{2n})$, $\tilde{G}_{2k}(\R^{2n+1})$ and $\tilde{G}_{2k+1}(\R^{2n+1})$, they respectively have
	\begin{align*}
	\frac{e^T_{S}}{e^N_{S}}&=\frac{1}{\prod_{i\in S}\prod_{j \not \in S}(\alpha^2_j-\alpha^2_i)} &
	\frac{e^T_{S_\pm}}{e^N_{S_\pm}}&=\frac{1}{\prod_{i\in S}\prod_{j \not \in S}(\alpha^2_j-\alpha^2_i)}\\
	\frac{e^T_{S_\pm}}{e^N_{S_\pm}}&=\frac{\pm \prod_{l\in S}\alpha_l}{\pm \prod_{l\in S}\alpha_l \prod_{i\in S}\prod_{j \not \in S}(\alpha^2_j-\alpha^2_i)} &
	\frac{\bar{e}^T_{S_\pm}}{e^N_{S_\pm}}&=\frac{\pm \prod_{l\not \in S}\alpha_l}{\pm \prod_{l\not\in S}\alpha_l \prod_{i\in S}\prod_{j \not \in S}(\alpha^2_j-\alpha^2_i)} 
	\end{align*}
	for ${G}_{2k+1}(\R^{2n+2}),\tilde{G}_{2k+1}(\R^{2n+2})$, they respectively have
	\begin{align*}
	\frac{\int r^T_S}{e^N_S} &=\frac{\prod_{l=1}^n\alpha_l \int_{S^1}\theta_{S^1}}{\prod_{l=1}^n\alpha_l \prod_{i\in S}\prod_{j \not \in S}(\alpha^2_j-\alpha^2_i)} &
	\frac{\int \tilde{r}^T_S}{e^N_S} &=\frac{\prod_{l=1}^n\alpha_l \int_{\R P^1}\theta_{\R P^1}}{\prod_{l=1}^n\alpha_l \prod_{i\in S}\prod_{j \not \in S}(\alpha^2_j-\alpha^2_i)}.
	\end{align*}
	All these fractions are equal to $\frac{1}{\prod_{i\in S}\prod_{j \not \in S}(\alpha^2_j-\alpha^2_i)}$. The difference by factor of $2$ comes from the fact that $\tilde{G}_{2k}(\R^{2n})$, $\tilde{G}_{2k}(\R^{2n+1})$ and $\tilde{G}_{2k+1}(\R^{2n+1})$ have fixed points indexed by $S_\pm$, while ${G}_{2k}(\R^{2n})$, ${G}_{2k+1}(\R^{2n+2})$ and $\tilde{G}_{2k+1}(\R^{2n+2})$ have fixed points or circles indexed by $S$.
\end{proof}

\begin{rmk}
	When the cohomological degree of a characteristic polynomial matches with the dimension of a Grassmannian, or equivalently $\sum_{j=1}^{k}j\cdot i_j=k(n-k)$, then the equivariant characteristic number will be a constant, i.e. an ordinary characteristic number. Moreover, we then get a formula of the ordinary characteristic numbers by substituting any $\alpha_i=a_i \in \R$ such that $a_i\not =0, a_i \not = \pm a_j$ into the localized expression of ABBV formula. For instance, we can choose $\alpha_i=i,\forall i$, or $\alpha_i=\sqrt{i},\forall i$. Moreover, we have the relations between ordinary Pontryagin characteristic numbers:
	\begin{align*}
	&\int_{\tilde{G}_{2k}(\R^{2n})}p^I=\int_{\tilde{G}_{2k}(\R^{2n+1})}e \cdot p^I=\int_{\tilde{G}_{2k+1}(\R^{2n+1})}\bar{e} \cdot p^I\\
	=&2\int_{{G}_{2k}(\R^{2n})}p^I=2\int_{{G}_{2k+1}(\R^{2n+2})}r\cdot p^I=2\int_{\tilde{G}_{2k+1}(\R^{2n+2})}\tilde{r}\cdot p^I.
	\end{align*}
\end{rmk}

\begin{rmk}
	Recall that localized equivariant Pontryagin class can be obtained from localized equivariant Chern class by replacing $\alpha_i$ with $\alpha^2_i$. Let $Sq:\Q[\alpha_1,\ldots,\alpha_n]$ be the ring homomorphism by sending $\alpha_i$ to $\alpha^2_i$ for every $i$. Then we have
	\[
	\int_{{G}_{2k}(\R^{2n})}(p^T)^I = Sq\big(\int_{{G}_{k}(\C^{n})}(c^T)^I\big).
	\]
	Since the ring homomorphism $Sq$ keeps rational numbers unchanged, we have the relation of ordinary Pontryagin numbers and Chern numbers
	\[
	\int_{{G}_{2k}(\R^{2n})}p^I = \int_{{G}_{k}(\C^{n})}c^I = \sum_{S\in \mathcal{S}} \frac{e^{i_1}_1(S)\cdots e^{i_k}_k(S)}{\prod_{i\in S}\prod_{j\not\in S} (j - i)}
	\]
	where the second identity is given in Cor \ref{thm:OrdChernNum}. 
\end{rmk}

\begin{rmk}
	Consider the $2$-covers of orientable Grassmannians $\pi: \tilde{G}_{2k}(\R^{2n}) \rightarrow {G}_{2k}(\R^{2n})$ and $\pi: \tilde{G}_{2k+1}(\R^{2n+2}) \rightarrow {G}_{2k+1}(\R^{2n+2})$. By the naturality of equivariant Pontryagin classes and the above relations among equivariant characteristic numbers, we see
	\[
	\int_{\tilde{G}_{2k}(\R^{2n})}\pi^*\big((p^T)^I\big)=\int_{\tilde{G}_{2k}(\R^{2n})}(p^T)^I = 2\int_{{G}_{2k}(\R^{2n})}(p^T)^I.
	\]
	From Prop \ref{prop:Pullr}, we have $ \pi^*{r}^T=2\tilde{r}^T$ , then
	\[
	\int_{\tilde{G}_{2k+1}(\R^{2n+2})}\pi^*\big({r}^T\cdot (p^T)^I\big) = 2\int_{\tilde{G}_{2k+1}(\R^{2n+2})}\tilde{r}^T\cdot (p^T)^I = 2\int_{{G}_{2k+1}(\R^{2n+2})}r^T\cdot (p^T)^I. 
	\]
\end{rmk}

\vskip 20pt
\bibliographystyle{amsalpha}

\end{document}